\documentclass[onefignum,onetabnum]{siamart171218}

\usepackage[utf8]{inputenc}
\usepackage{lipsum}
\usepackage{epstopdf}
\usepackage{algorithmic}
\usepackage{listings}
\usepackage{tikz}
\usetikzlibrary{matrix}
\usepackage{tcolorbox}
\usepackage{bbding}
\usepackage{amsopn}
\usepackage{subcaption}
\usepackage{multirow}

\usepackage{array}
\usepackage{url}
\usepackage{amsmath,amssymb,amsbsy}
\usepackage{paralist}
\usepackage{xcolor}
\usepackage{color}
\usepackage{graphicx}
\graphicspath{{./figs/}}
\usepackage{algorithm}
\usepackage{algorithmic}
\usepackage{comment}
\usepackage{wrapfig}
\usepackage{caption}
\usepackage{subcaption}
\usepackage{bm}
\usepackage{framed}
\usepackage{fancyhdr}
\usepackage{cite}
\usepackage{cleveref}
\usepackage{mathrsfs}
\usepackage{multirow}\usepackage{makecell}

\newtheorem{eg}{Example}

\newtheorem{assum}{Assumption}
\newtheorem{cond}{Condition}

\newcommand{\R}{\mathbb{R}}

\newcommand{\be}{\begin{equation}}
\newcommand{\ee}{\end{equation}}
\newcommand{\en}{\begin{equation*}}
\newcommand{\een}{\end{equation*}}
\newcommand{\eqn}{\begin{eqnarray}}
\newcommand{\eeqn}{\end{eqnarray}}
\newcommand{\bmat}{\begin{bmatrix}}
\newcommand{\emat}{\end{bmatrix}}
\newcommand{\btab}{\begin{tabular}}
\newcommand{\etab}{\end{tabular}}

\newcommand{\Mcal}{\mathcal{M}}

\newcommand{\sign}{\operatorname{sign}}

\newcommand{\supp}{\operatorname{supp}}

\DeclareMathOperator*{\argmin}{\text{argmin}}

\newcommand{\prox}{\operatorname{prox}}

\def \st {\operatorname*{subject\ to\ }}
\newlength{\imgwidth}
\newcommand{\revise}[1]{\textcolor{black}{#1}}

\newboolean{twoColVersion}
\setboolean{twoColVersion}{false}
\newcommand{\twoCol}[2]{\ifthenelse{\boolean{twoColVersion}} {#1} {#2} }

\usepackage[framemethod=tikz]{mdframed}

\ifpdf
  \DeclareGraphicsExtensions{.eps,.pdf,.png,.jpg}
\else
  \DeclareGraphicsExtensions{.eps}
\fi

\newcommand{\iprod}[2]{\left \langle #1, #2 \right \rangle }

\newsiamremark{remark}{Remark}
\newsiamremark{hypothesis}{Hypothesis}
\crefname{hypothesis}{Hypothesis}{Hypotheses}
\newsiamthm{claim}{Claim}

\newcommand{\rmnum}[1]{\uppercase\expandafter{\romannumeral #1}}

\headers{On the Convergence of the semismooth Newton method}{JIANG HU, TONGHUA TIAN, SHAOHUA PAN, AND ZAIWEN WEN} 

\title{On the local convergence of the Semismooth Newton Method for composite optimization}
\author{Jiang Hu\thanks{Massachusetts General Hospital and Harvard Medical School, Harvard University, Boston, MA 02114
(\email{hujiangopt@gmail.com}).}
	\and
Tonghua Tian\thanks{School of Operations Research and Information Engineering,  Cornell University, Ithaca, NY14853 USA (\email{tt543@cornell.edu}).} 
	\and
	Shaohua Pan\thanks{School of Mathematics, South China University of Technology, Guangzhou, China (\email{shhpan@scut.edu.cn}).}
	\and
	Zaiwen Wen\thanks{Beijing International Center for Mathematical Research, Center for Data Science and College of Engineering, Peking University, Beijing, China
	(\email{wenzw@pku.edu.cn}).}
	}

\begin{document}

\maketitle

\begin{abstract}
In this paper, we consider a large class of nonlinear equations derived from first-order type  methods  for solving composite optimization problems. Traditional approaches to establishing superlinear convergence rates of semismooth Newton-type methods for solving nonlinear equations usually postulate either nonsingularity of the B-Jacobian or smoothness of the equation. We investigate the feasibility of both conditions. For the nonsingularity condition, we present equivalent characterizations in broad generality, and illustrate that they are easy-to-check criteria for some examples. For the smoothness condition, we show that it holds locally for a large class of residual mappings derived from composite optimization problems. Furthermore, we investigate a relaxed version of the smoothness condition - smoothness restricted to certain active manifolds. We present a conceptual algorithm utilizing such structures and prove that it has a superlinear convergence rate. 
\end{abstract}


\section{Introduction}
In this paper, we study the semismooth Newton (SSN) method for solving composite optimization problem:
\be \label{prob} \min_{x \in \R^n}\;\; \psi(x) = f(x) + h(x), \ee
where $f,h\!:\mathbb{R}^n\to\bar{\mathbb{R}}:=(-\infty,\infty]$ are proper, closed, and prox-regular functions. In many scenarios, see, e.g., \cite{fukushima1981generalized,eckstein1992douglas,zhao2010newton,li2016douglas}, the optimality condition of \eqref{prob} or its associated subproblems can be represented as the solution of a nonlinear equation
\be \label{prob:nonsmootheq} F(x) = 0, \ee
where $F:\mathbb{R}^n\to\mathbb{R}^n$ is locally Lipschitz continuous. By requiring certain properties on $f$ and $h$, the resulted mapping $F$ is semismooth.

Originally, the SSN method for \eqref{prob:nonsmootheq} is developed for general semismooth mappings in  \cite{mifflin1977semismooth,pang1993nonsmooth,qi1993convergence,qi1993nonsmooth}.  In the last decades, the SSN method is designed to obtain the stationary points of the structured composite optimization problem \eqref{prob} through the augmented Lagrangian framework (ALM) or seeking a root of a constructed semismooth residual mapping. The efficiency of SSN augmented Lagrangian methods has been verified in sparse optimization \cite{li2018highly,lin2019efficient}, linear programming \cite{li2020asymptotically}, and semidefinite programming \cite{zhao2010newton,yang2015sdpnal+}.
 The SSN methods based on the residual mappings of the stationarity of \eqref{prob} induced by the first-order methods, e.g., the proximal gradient method (PGM) \cite{xiao2018regularized,milzarek2014semismooth,milzarek2019stochastic,yang2021stochastic} and the Douglas–Rachford splitting (DRS) \cite{li2018semismooth,liu2022multiscale}, 
have also been shown to be very competitive in solving the structured composite optimization problems.

The local superlinear or quadratic convergence of the SSN method is typically established by assuming the invertibility of all elements of the B-Jacobian at the limiting point of the iterates (which is the so-called BD-regularity condition). The BD-regularity condition of \eqref{prob:nonsmootheq} is closely related to the second-order sufficient condition of \eqref{prob}. For smooth $f$ and certain convex $h$, it is proved in \cite[Theorem 5.4.4]{milzarek2016numerical} that the second-order sufficient condition along with the strict complementarity condition implies the BD-regularity of the natural residual mapping induced by the PGM. As the sufficient conditions may not be enough to fully characterize BD-regularity, we present conditions that are both sufficient and necessary for the BD-regularity of the residual mappings induced by the PGM and DRS method. To the best of our knowledge, this characterization covers the existing ones for some special composite optimization problems such as the Lasso problem \cite{griesse2008semismooth,milzarek2014semismooth}.

Note that BD-regularity is a strong condition since it implies the isolatedness of the stationary point, which may not hold in practice. However, in most existing works, this condition is indispensable to establishing the superlinear convergence of the SSN method for nonsmooth mappings $F$. By contrast, in the cases where $F$ is locally smooth, the error bound (EB) condition can often serve as a weaker substitute for BD-regularity in proving the superlinear convergence of the Levenberg–Marquardt (LM) method  \cite{fan2004inexact} and the Newton method \cite{zhou2005superlinear,zhou2006convergence} to solve \eqref{prob:nonsmootheq} (while the conditions for $F$ being locally smooth is  not clear). A regularized Newton method for solving a smooth and monotone gradient system without nonsingluarity assumption on the Hessian is also analyzed in \cite[Section 3]{li2004regularized}. In \cite{zhou2017unified}, the authors prove that the EB condition holds for many structured convex composite optimization problems. By assuming the EB condition, the superlinear convergence of a regularized proximal Newton method for solving \eqref{prob} with twice continuously differentiable $f$ and convex $g$ is presented in \cite{yue2019family}. The core tool used to establish the superlinear convergence in  \cite{fan2004inexact,zhou2005superlinear,zhou2006convergence,yue2019family} is the high-order approximation property of the subproblems. To be specific, for the smooth $F$, both the Newton method and the LM method directly utilize the Taylor expansion of $F$ to construct the subproblems. For the nonsmooth $F$ from \eqref{prob} with twice continuously differentiable $f$ and convex $g$, the high-order approximate subproblems in the proximal Newton method are constructed by using the Taylor expansion of the smooth part of $F$ while keeping the remaining nonsmooth part.
Note that solving the proximal Newton subproblem can be still difficult since the nonsmooth part is preserved. The SSN method aims to construct a tractable subproblem (i.e., a linear system) by investigating the linear approximation of the mapping $F$.
Hence, the smoothness of $F$ is essential for establishing the superlinear convergence rate of the SSN method in the absence of the BD-regularity. It is worth noting that the local superlinear convergence of Newton-type methods with more complicated subproblems, e.g., linear program, and the LM-type methods are also considered in \cite{fischer2002local,facchinei2014lp,fischer2016convergence,kummer1988newton,yamashita2001rate,fischer2010inexactness} to tackle the nonsmoothness of $F$. Our second goal is to establish the local smoothness of a large class of residual mappings derived from the composite optimization problems and  elaborate the superlinear convergence of the associated SSN method under the EB.

The contributions can be summarized as follows. 1)
 For the residual mappings induced by the PGM, the DRS, and the ALM method, we present equivalent characterizations of the BD-regularity condition. To our best knowledge, this is the first result on the sufficient and necessary conditions for BD-regularity. The Lasso problem is presented as a concrete example to enhance the understanding. 
    2) We provide two sufficient conditions for the local smoothness of the residual mapping $F$ around a stationary point. Our first condition requires the strict complementarity (SC) condition of problem \eqref{prob} and the partial smoothness of $f$ and $h$. We show the local smoothness of the proximal operators under the partial smoothness condition, which extends the smoothness result on the projection operators of smooth manifolds \cite{lewis08projection}.
    Our second condition, based on the analysis in \cite{rockafellar1988first,stella2017forward}, consists of the twice epi-differentiability, the generalized quadratic property, and the closedness of the set of nondifferentiable points of the proximal operators. Moreover, the generalized quadratic property is satisfied by fully decomposable functions under the SC. Various norm functions and indicator functions are presented as concrete examples that satisfy the proposed two conditions. Under either condition, the SSN method reduces to the smooth Newton method locally. 
    3) We show that a projected SSN method has locally  superlinear convergence assuming the local smoothness of $F$ on a submanifold, the EB condition, and certain local monotonicity. Such result includes the SSN method as a special case. In particular, the local superlinear convergence of the SSN method can be obtained if the SC condition, the EB condition, the local monotonicity, and either of the two proposed sufficient smoothness conditions holds. 
    Numerical experiments are performed to demonstrate our findings. A summary on comparisons with existing works is presented in Table \ref{tab:comparison}.
    Finally, our theoretical results are easily generalizable to problem \eqref{prob} with $f$ and $h$ defined on any Euclidean space (e.g., the matrix space) instead of $\R^n$.

\begin{table}[t] 
\caption{The local convergence rate of the (semismooth) Newton methods. The column ``isolatedness'' indicates if the stationary point is isolated.} \label{tab:comparison}
	\centering
	\setlength{\tabcolsep}{1pt}
\begin{tabular}{|c|c|p{3cm}|c|c|}
	\hline 
	$F$ & method & regularity condition & isolatedness &\quad  local convergence rate  \\
	\hline 
	smooth & Newton  & \quad nonsingularity & necessary & \qquad quadratic \cite{wright1999numerical,zhou2005superlinear,zhou2006convergence} \\ \cline{1-5}
	semismooth & SSN & \quad BD-regularity & necessary & \quad superlinear \cite{pang1993nonsmooth,qi1993convergence,qi1993nonsmooth}  \\\hline
	smooth & LM  & \quad  \qquad EB & not necessary &  \qquad quadratic \cite{fan2004inexact} 
% 	and \cite{li2004regularized} (regularized Newton method for a smooth and monotone gradient system) 
	\\ \cline{1-5}
	semismooth & SSN & EB + smoothness on active manifolds
	& not necessary & superlinear (\textbf{this work})  \\\hline
\end{tabular}
\end{table}

\subsection{Notation}
For any positive integer $n$, we define $[n]=\{1,2,\ldots, n\}$. The support set of a vector $x\in \R^n$ is defined as ${\rm supp}(x):=\{i: x_i \ne 0\}$. With a slight abuse of notation, we use $\|x\|$ to denote the $\ell_2$ norm of a vector $x$ and $\|X\|$ to denote the spectral norm of a matrix $X$. Let $B(x,r)$ be the open ball with radius $r>0$ centered at $x \in \R^n$, i.e. $B(x,r):=\{ y\in \R^n :\|y-x\| < r\}$, and let $\bar{B}(x, r)$ be the closure of $B(x,r)$. For a matrix $A \in \R^{m\times n}$, we denote its range space and null space by ${\rm Range(A)} = \{Ad: d \in \R^n\}$ and
${\rm Ker}(A) = \{d \in \R^n : Ad = 0\}$, respectively. For a symmetric matrix $A \in \R^{n\times n}$, we let $\lambda_{\max}(A)$ and $\lambda_{\min}(A)$ denote its largest and smallest eigenvalues respectively. Let $\mathrm{id}$ denote the identity mapping from $\R^n$ to $\R^n$. For a set $S \subset \R^n$, we use $S^c$ to denote its complement in $\R^n$. The indicator function $\delta_{S}$ is defined as $\delta_{S}(x) = 0$ if $x \in S$ and $+\infty$ otherwise. When $S$ is convex, let $\mathrm{ri}\, S$ denote its relative interior, and let $\mathrm{par}\, S$ denote the subspace parallel to its affine hull.

\subsection{Organization}
We begin with some preliminaries on the residual mappings and the SSN method in Section \ref{sec:preli-ssn}. The equivalent characterizations of the BD-regularity condition are investigated in Section \ref{sec:bd-regular}. We introduce the SC condition and the two sufficient conditions for the local smoothness of the residual mappings in Section \ref{sec:sc}. 
The analysis of the local superlinear convergence rate of a projected SSN method is presented in Section \ref{sec:convergence}. Finally, we show the numerical verification in Section \ref{sec:num}.

\section{Preliminaries} \label{sec:preli-ssn}
\subsection{Semismoothness}
Let us first recall the definition of the B-Jacobian. 
By the Rademacher's theorem, a locally Lipschitz mapping 
$F:\mathbb{R}^n\to\mathbb{R}^n$ is almost everywhere differentiable. Denote by $D_F$ the set of differentiable points of $F$. The B-Jacobian  of $F$ at $x$ is defined as
$$
\partial_{B} F(x):=\left\{\lim _{k \rightarrow \infty} J\left(x^{k}\right) \mid x^{k} \in D_{F}, x^{k} \rightarrow x\right\},
$$
where $J(x)$ denotes the Jacobian of $F$ at $x \in D_F$. Obviously, $\partial_{B} F(x)$ may not be a singleton. We say $F$ is BD-regular at $x$ if each $J \in \partial_B F(x)$ is nonsingular.  The Clarke subdifferential of $F$ at $x$ is defined as
\[ \partial F(x) = \mathrm{conv}\left(\partial_{B} F(x)\right),\]
where $\mathrm{conv}(S)$ represents the convex hull of $S$. The mapping $F$ is said to be semismooth at $x$ if
\begin{itemize}
    \item[(a)] $F$ is directionally differentiable at $x$;
    \item[(b)] for all $d$ and $J \in \partial F(x+d)$, it holds that
\[ \| F(x+d) -  F(x) - Jd \| = o(\|d\|), \;\; \text{as } d \rightarrow 0. \] 
\end{itemize} 
Moreover, $F$ is said to be strongly semismooth at $x$ if the equation in (b) is replaced by $\| F(x+d) -  F(x) - Jd \| = O(\|d\|^2)$. We say that $F$ is semismooth if $F$ is semismooth at every $x \in \R^n$.

\subsection{The proximal mapping} The nonlinear equation \eqref{prob:nonsmootheq} discussed in this paper for  \eqref{prob} is constructed from the proximal mappings of $f$ and $h$. Given a proper closed function $\phi: \mathbb{R}^n \rightarrow \bar{\mathbb{R}}$ and a constant $t > 0$, the proximal mapping $\mathrm{prox}_{t \phi}$ is defined as
\begin{equation}
    \mathrm{prox}_{t \phi}(y) {\,=\,}\argmin_{x\in\mathbb{R}^n} \left\{ \phi(x) + \frac{1}{2 t}\|x-y\|_2^2 \right\}, \quad \forall y \in \mathbb{R}^n.
\end{equation}
In general, $\mathrm{prox}_{t \phi}$ may be empty or multi-valued. In this paper, we restrict our attention to functions whose proximal mappings locally are single-valued functions. We say that $\phi$ is prox-bounded \cite[Definition 1.23]{rockafellar2009variational} if there exist $t > 0$ and $y \in \R^n$ such that $\inf_{x\in\mathbb{R}^n} \{ \phi(x) + \frac{1}{2t}\|x-y\|^2 \} > -\infty$. 

Given a point $x$ where $\phi(x)$ is finite, we call $v \in \R^n$ a regular subgradient \cite[Definition 8.3]{rockafellar2009variational} of $\phi$ at $x$, written as $v \in \hat{\partial} \phi(x)$, if
\[ \phi(y) \geq \phi(x) + \iprod{v}{y-x} + o(\|y - x\|), \;\; \forall y \in \R^n. \]
We call $v$ a limiting subgradient \cite[Definition 8.3]{rockafellar2009variational}, written as $v\in \partial \phi(x)$, if there are sequences $x_k \rightarrow x$ with $\phi(x_k) \rightarrow \phi(x)$ and $v_k \in \hat{\partial} \phi(x_k)$ with $v_k \rightarrow v$, and call $v$ a horizon subgradient, written as $v \in \partial^{\infty} \phi(x)$, if $x_k \rightarrow x$ with $\phi(x_k) \rightarrow \phi(x)$ and $v_k \in \hat{\partial} \phi(x_k)$ with $t_k v_k \rightarrow v$ for some $t_k \searrow 0$. 
For prox-regular functions \cite[Definition 13.27]{rockafellar2009variational}, the regular subgradients and subgradients coincide.

\begin{definition}[Prox-regular function] \label{def:prox-regular}
 A function $\phi: \mathbb{R}^n \rightarrow \bar{\mathbb{R}}$ is prox-regular at $\bar{x}$ for $\bar{v}$ if $\phi$ is finite and locally lower semicontinuous at $\bar{x}$ with $\bar{v}\in\partial\phi(\bar{x})$, and there exist constants $\varepsilon > 0$ and $\rho \geq 0$ such that 
 \begin{align} \label{eq:prox-regular} 
  \phi(y) \geq  \phi(x) + \iprod{v}{y-x} - \frac{\rho}{2}\|y-x\|^2, \;\; \forall y\in B(\bar{x},\varepsilon)
  \end{align}
 whenever $v\in\partial\phi(x)$, $\|v-\bar{v}\| < \varepsilon$, $\|x - \bar{x}\| < \varepsilon$, $\phi(x) < \phi(\bar{x}) + \varepsilon$.
We say that $\phi$ is prox-regular at $\bar{x}$ if $\phi$ is prox-regular at $\bar{x}$ for every $\bar{v} \in \partial\phi(\bar{x})$. We say $\phi$ is prox-regular if it is prox-regular at every point of its domain.
\end{definition}

Note that every proper closed convex function is prox-regular \cite[Example 13.30]{rockafellar2009variational}, as well as every $C^2$ smooth function. The concept of prox-regularity is important in studies concerning proximal mappings, due to the following fact \cite[Proposition 13.37]{rockafellar2009variational}.

\begin{theorem} \label{prop:lip-prox-regular}
If $\phi: \mathbb{R}^n \rightarrow \bar{\mathbb{R}}$ is prox-bounded and prox-regular at $\bar{x}$ for $\bar{v}$, then for all sufficiently small $t > 0$, there is a neighborhood of $\bar{x} + t \bar{v}$ on which $\mathrm{prox}_{t \phi}$ is monotone, single-valued and Lipschitz continuous.
\end{theorem}

\subsection{Semismooth systems from  composite problem \eqref{prob}}

There are different ways to construct the nonlinear equation \eqref{prob:nonsmootheq}. We briefly summarize three systems induced by PGM, DRS \cite{xiao2018regularized,li2018semismooth} and ALM \cite{li2020asymptotically,zhao2010newton}, respectively. 

\subsubsection{Natural residual and DRS residual}

Let us first consider the basic step of PGM proposed in \cite{fukushima1981generalized,lions1979splitting,combettes2011proximal,beck2009fast} for solving \eqref{prob} when $f$ is smooth, where, unless otherwise specified, by ``smooth'' we always mean $C^1$ smooth. In the $k$-th iteration, it  updates
\be \label{eq:pgm-it} x_{k+1}  \in \mathrm{prox}_{th}(x_{k} - t\nabla f(x_k)), \ee
where $t > 0$ is a step size. Since $h$ is required to be prox-regular, the proximal mapping ${\rm prox}_{th}$ is single-valued and Lipschitz continuous for a properly chosen $t$ by Proposition \ref{prop:lip-prox-regular}. Hence the iterative scheme \eqref{eq:pgm-it} can be seen as a fixed-point  method to solve the following equation:
\begin{equation} \label{eq:pgm} F_{{\rm PGM}}(x) := x - \mathrm{prox}_{th}(x - t \nabla f(x)) = 0. \end{equation}
We call $F_{\rm PGM}$ the natural residual. 

%\subsubsection{Douglas-Rachford splitting residual}

Another splitting method for solving \eqref{prob} is the DRS method \cite{eckstein1992douglas,li2016douglas,davis2016convergence}, where both $f$ and $h$ can be nonsmooth and nonconvex.
 In the $k$-th iteration, the DRS scheme is 
 \be \label{eq:drs-it} 
 \left\{
 \begin{aligned}
 x_{k+1} & \in \mathrm{prox}_{th}(z_k), \\
 y_{k+1} & \in \mathrm{prox}_{tf}(2x_{k+1} - z_k), \\
 z_{k+1} & = z_k + y_{k+1} - x_{k+1}.
 \end{aligned}
 \right.
 \ee 
 Similar to PGM, when $t$ is chosen properly to make $\mathrm{prox}_{th}$ and $\prox_{tf}$ single-valued and Lipschitz continuous close to points of interest, the DRS method can be regarded as a fixed-point procedure to solve the following equation:
  \begin{equation} \label{eq:drs}  F_{{\rm DRS}}(z) := \mathrm{prox}_{th}(z) - \prox_{tf}(2\mathrm{prox}_{th}(z) - z) = 0. \end{equation}
We call $F_{{\rm DRS}}$ the DRS residual.

Since the proximal mappings are locally Lipschitz continuous for sufficiently small $t > 0$, the residual mappings defined in \eqref{eq:pgm} and \eqref{eq:drs} are also locally Lipschitz continuous. When the proximal mappings ($\prox_{th}$ for \eqref{eq:pgm}, both $\prox_{tf}$ and $\prox_{th}$ for \eqref{eq:drs}) are furthermore semismooth, one can easily verify the semismoothness of $F_{\mathrm{PGM}}$ and $F_{\mathrm{DRS}}$ as well.

Let us clarify the relationship between the stationarity induced by \eqref{eq:pgm} and \eqref{eq:drs} and the classic notion of stationarity defined as follows. 

\vskip 2mm

\begin{definition} For problem \eqref{prob}, we say that $x \in \mathrm{dom}\,\psi$ is stationary if $0 \in \partial f(x) + \partial h(x)$.
\end{definition}

\vskip 2mm
 If $f$ or $h$ is locally Lipschitz, we have from \cite[Exercise 10.10]{rockafellar2009variational} that $\partial \psi(x) \subset \partial f(x)+\partial h(x)$. 
Firstly, consider the case where $f$ is smooth. Based on the definition of the proximal mapping and \cite[Theorem 10.1]{rockafellar2009variational}, every root of the natural residual \eqref{eq:pgm} is stationary, and every root $z$ of the DRS residual satisfies that $\mathrm{prox}_{th} (z)$ is stationary.
In the case where $f$ is nonsmooth, by the definition of the proximal mapping, we can also deduce that $\prox_{th}(z)$, with $z$ being a root of the DRS residual, is stationary. 
 Hence, assuming semismoothness of the proximal mappings, one can apply the SSN method to solve for the roots of both $F_{\rm PGM}$ and $F_{\rm DRS}$ to get the stationary points of \eqref{prob}. From the computational aspect, the residual mapping $F_{\rm PGM}$ is preferred in designing the SSN method for smooth $f$ (as the computation cost of $\nabla f$ is generally much lower than that of ${\rm prox}_{tf}$). For nonsmooth $f$, one needs to use $F_{\rm DRS}$ as $F_{\rm PGM}$  is undefined.

\subsubsection{Gradient mappings of the ALM subproblem} \label{subsubsec:alm}
Consider the case that $f$ is smooth and convex and $h$ is convex. Then, their proximal mappings associated to any $t > 0$ are single-valued and globally Lipschitz continuous with modulus $1$. The SSN augmented Lagrangian methods \cite{zhao2010newton,yang2015sdpnal+,li2018highly} have exhibited superior performance on semidefinite programming and Lasso. The general idea is to split the composite structure by introducing an extra variable and then applying the augmented Lagrangian method, whose each iteration consists of a joint minimization for the primal variables and a gradient ascent step for the dual variables. Specifically, the augmented Lagrangian function for \eqref{prob} is 
\[ L_{t}(x,z) = f(x) + e_{t h}(x - t z) - \frac{t}{2}\|z\|^2, \]
where $e_{t h}(x):=\min_{y \in \R^n}\; h(y) + \frac{1}{2t}\|y-x\|^2$ is the Moreau envelope of $h$.
For a given $z_0$ and a sequence of positive and nonincreasing parameters $\{t_k\}$, the basic iterate steps of the augmented Lagrangian method are  
\be \label{alg:alm}
\left\{ \begin{aligned}
x_{k+1} & = \argmin_{x \in \R^n} \;\varphi(x;t_k, z_k) := L_{t_k}(x, z_k), \\
z_{k+1} & = z_k + \frac{1}{t_k} ({\rm prox}_{t_k h}(x_{k+1} - t_k z_k) - x_{k+1}). 
\end{aligned} \right.
\ee
For the $x$-subproblem, the objective function $\varphi(\cdot; t_k, z_k)$ is convex and smooth with gradient 
\be \label{eq:res-alm}\nabla \varphi(x;t_k, z_k) = \nabla f(x) +  \frac{1}{t_k}(x- t_k z_k - \prox_{t_k h}(x - t_k z_k)).  \ee
Given any fixed constant $t > 0$ and parameter $z \in \R^n$, consider the system
\begin{equation} \label{eq:gradient} 
F_{\rm ALM} (x; z) :=\nabla f(x) +  \frac{1}{t}\big(x- t z - \prox_{t h}(x - t z)\big) = 0. \end{equation}
Then, $x_{k+1}$ is a root of the system \eqref{eq:gradient} with $t = t_k$ and $z = z_k$.
For structured $h$ \cite{zhao2010newton,yang2015sdpnal+,li2018highly}, the gradient mapping  $F_{{\rm ALM}} (\,\cdot\,; z)$ is also semismooth, where the SSN method can be employed. With a slight abuse of notation, we sometimes omit the parameter $z$, and simply write the mapping $F_{{\rm ALM}} (\,\cdot\,; z)$ as $F_{{\rm ALM}}$.

\subsection{Semismooth Newton method}
We describe a general framework of the SSN method for solving equation \eqref{prob:nonsmootheq}. In the $k$-th iteration, we choose an element $J_k$ from $\partial F(x^k)$ and solve the following linear equation
\be \label{eq:ssn} (J_k + \mu_k I)d = - F_k + r_k, \ee
where $F_k = F(x_k)$, $\mu_k \geq 0$ is a shift parameter to ensure certain nonsingularity of $J_k + \mu_k I$, and $r_k$ is the error vector to measure the inexactness. In order to achieve fast convergence, we choose
\be \label{eq:muk} \mu_k = \| F_k\|, \;\; \text{for } k=1, 2, \ldots. \ee
 Note that similar setting of $\mu_k$ is also employed in the LM method for solving \eqref{prob:nonsmootheq} with smooth $F$ \cite{fan2004inexact}.
Denote by $d_k$ the solution of equation \eqref{eq:ssn}. The  (regularized inexact) SSN step is
\be \label{eq:ssn-it} x_{k+1} = x_k + d_k. \ee
For globalization, we need to combine either a line search procedure \cite{zhao2010newton}, or the proximal gradient step or other steps with decrease guarantees (see, e.g., \cite{xiao2018regularized,milzarek2019stochastic,yang2021stochastic}) depending on the specific form of $F$.  Extensive numerical experiments show that the SSN step enjoys fast local convergence. 

\section{Equivalent characterizations of BD-regularity}  \label{sec:bd-regular}

Let $x^*$ be a root of a locally Lipschitz mapping $F$. By \cite[Lemma 2.6]{qi1993convergence}, if the BD-regularity of $F$ holds at $x^*$, then there exists a scalar $c > 0$ and a scalar $\delta > 0$ such that $\|V^{-1}\| \leq c$ for all $V \in \partial_B F(x)$ with $\|x - x^*\| \leq \delta$. Let $\{x_k\}$ be the sequence generated by \eqref{eq:ssn-it} with $J_k \in \partial_B F(x^k)$. 
 If $x_k$ satisfies $\|x_k -x^*\| \leq \delta$, the local superlinear convergence is obtained by assuming the semismoothness of $F$ and $\|r_k\| = o(\|x_k - x^*\|)$. 
According to this result, the verification of  the BD-regularity is important. Hence, we  present equivalent characterizations of the BD-regularity condition for the natural residual \eqref{eq:pgm}, the DRS residual \eqref{eq:drs} and the gradient mapping \eqref{eq:gradient}, and then discuss the condition for the invertibility of the shifted B-Jacobian.

\subsection{Characterization for the natural residual}
In this subsection, we consider the case where $f$ is $C^2$ smooth and $h$ is convex.
For the natural residual \eqref{eq:pgm}, its B-Jacobian at $x$ with nonsingular $I - t \nabla^2\!f(x)$ \cite[Lemma 1]{chan2008constraint} takes the form of
\[ \partial_{B} F_{\mathrm{PGM}} (x) =\{ I - M (I - t\nabla^2 f(x)): M \in \partial_B \mathrm{prox}_{th}(x - t \nabla f(x)) \}. \]
The BD-regularity of $F_{\mathrm{PGM}}$ holds at $x$ if and only if all elements of $\partial_{B} F_{\mathrm{PGM}} (x)$ are invertible. It is shown in \cite{bauschke2011convex,rockafellar2009variational} that $\mathrm{prox}_{th}$ is monotone and nonexpansive. Hence, by \cite[Lemma 3.3.5]{milzarek2016numerical}, $I \succeq M \succeq 0$ for all $M \in \partial_B \mathrm{prox}_{th}(y)$ and $y \in \mathbb{R}^n$, where $A \succeq B$ means $A-B$ is positive semidefinite. For a general $h$, it is hard to verify the BD-regularity of $F_{\mathrm{PGM}}$, but for some structured $h$, we can present an equivalent characterization of BD-regularity using second-order information of $f$. 

\begin{proposition} \label{lemma:BD-regularity}
 Fix any $x \in \mathbb{R}^n$. Suppose that $\nabla^2\!f(x)$ is positive semidefinite, $h$ is convex, and $0 < t < \frac{1}{\lambda_{\max}(\nabla^2 f(x))}$. Then, the BD-regularity of $F_{\mathrm{PGM}}$ holds at $x$ if and only if $\nabla^2\!f(x)$ is positive definite on the subspace $\mathrm{Ker}(I-M)$ for all $M \in \partial_B \mathrm{prox}_{th}(x - t \nabla f(x))$.
\end{proposition}
\begin{proof}
 $\Longleftarrow$. 
 Suppose the BD-regularity does not hold. Then there exist a nonzero vector $d \in \R^n$ and $M \in \partial_B \mathrm{prox}_{th}(x - t \nabla f(x))$ satisfying
\be \label{eq:bd-1} \left[I - M (I - t\nabla^2 f(x)) \right]d = (I - M)d + tM \nabla^2 f(x)d = 0. \ee
Let $e = (I - t\nabla^2 f(x)) d$. We have from \eqref{eq:bd-1} that $d = M e$ and
\begin{equation}\label{eq:bd-2}
    (I - M)M e + tM \nabla^2 f(x) M e = 0.
\end{equation}
By \cite[Lemma 3.3.5]{milzarek2016numerical}, $e^{\top} (I - M)M e \geq 0$, 
while $te^{\top } M \nabla^2 f(x) M e \ge 0$ is implied by the positive semidefiniteness of $\nabla^2 f(x)$. Thus, equation \eqref{eq:bd-2} implies that
\begin{equation*}
    (I - M)M e = M \nabla^2 f(x) M e = 0,
\end{equation*}
which means that $d \in \mathrm{Ker}(I-M)$ and  $d^{\top} \nabla^2 f(x) d = 0$. This shows that $\nabla^2 f(x)$ is not positive definite on $\mathrm{Ker}(I-M)$, a contradiction to the give condition.

 $\Longrightarrow$. Suppose the implication does not hold. There exist $M\in\partial_B \mathrm{prox}_{th}(x\!-\!t\nabla f(x))$ and nonzero $d\in{\rm Ker}(I\!-\!M)$ such that $\nabla^2 f(x) d = 0$. Then, noting $(I\!-\!M)d=0$, we have
\begin{align*}
     (I\!-\!M)d+tM\nabla^2f(x)d &= tM\nabla^2f(x)d = 0.
 \end{align*}
 We obtain a contradiction.
\end{proof}

\begin{eg} \label{eg}
    Consider the case $h(x) = \lambda \|x\|_1$.  Let $x^*$ be a stationary point.  It is easy to check that each  element of $\partial_B \mathrm{prox}_{th}(x^* - t \nabla f(x^*))$ is a diagonal matrix $M$ with
    \be \label{eq:jaco-l1} M_{ii} \begin{cases}
        = 1, \;\; & x_i^* \ne 0,\\
        = 0, \;\; & x_i^* = 0 \; \mathrm{and} \; [|\nabla f(x^*)|]_i < \lambda, \\
        \in \{0,1\}, \;\; & x_i^* = 0\; \mathrm{and} \; [|\nabla f(x^*)|]_i = \lambda,
    \end{cases} \ee
    where both $0$ and $1$ can be attained in the last case.
    Then $\mathrm{Ker}(I-M) = \mathrm{Range}(M)$ for each $M \in \partial_B \mathrm{prox}_{th}(x^* - t \nabla f(x^*))$ and
    \begin{equation*}
        \bigcup_{M \in \partial_B \mathrm{prox}_{th}(x^* - t \nabla f(x^*))} \mathrm{Range}(M) = \{v \in \mathbb{R}^n : v_i = 0 ~\text{for all}~ i \notin T\},
    \end{equation*}
    where $T := \{ i \in [n] : x_i^* \ne 0 \} \cup \{ i \in [n]: x_i^* = 0, [|\nabla f(x^*)|]_i = \lambda \}$.
    Hence applying Proposition \ref{lemma:BD-regularity}, we can see that the BD regularity of $F_{\rm PGM}$ holds at $x^*$ if and only if $ \left[\nabla^2 f(x^*)\right]_{TT} \succ 0$, where $A_{TT}$ denotes a submatrix consisting of rows and columns in $T$ of $A$. 
 \end{eg}

\begin{remark} \label{rem:BD-PGM}
In Example \ref{eg}, the positive definiteness of $[\nabla^2 f(x^*)]_{TT}$ corresponds to the strong second-order sufficient condition \cite[Example 5.3.12]{milzarek2016numerical}. Similar results on the characterizations of the BD-regularity have been shown in \cite[Proposition 3.11]{griesse2008semismooth} and \cite[Lemma 4.7]{milzarek2014semismooth}. 
For general $h$, it is proved in \cite[Theorem 5.4.4]{milzarek2016numerical}
that the second-order sufficient condition together with SC implies BD-regularity. When reduced to the $\ell_1$ norm case, these conditions read $[\nabla^2 f(x^*)]_{TT} \succ 0$ and $\{ i \in [n] : x_i^* =0, |[\nabla f(x^*)]_i| = \lambda \} = \emptyset$, which are stronger than that of Proposition \ref{lemma:BD-regularity}. 
\end{remark}
\begin{remark}
For general $h$, e.g. the Euclidean norm, the nuclear norm, the indicator functions of polyhedral sets and the simplex, and spectral functions, since the union of $\mathrm{Ker}(I-M)$ is more complex, its characterization of BD-regularity will be more difficult to interpret. One can refer to \cite{patrinos2014forward} for the expressions of the proximal mappings of various $h$.
\end{remark}

\subsection{Characterization for the DRS residual}
 In this subsection, we consider the case where both $f$ and $h$ are nonsmooth and prox-regular. For the DRS residual \eqref{eq:drs}, from \cite[Theorem 2.6.6]{clarke1983nonsmooth}, its Clarke Jacobian at any $z\in\mathbb{R}^n$ has the following property
\begin{equation}\label{BJac-DRS} 
 \partial F_{\mathrm{DRS}}(z)v = \Big\{\big[M-D(2M-I)\big]v\,:\, M \in \partial\mathrm{prox}_{th}(z), \, D = \nabla \mathrm{prox}_{tf}(y)|_{y = 2\mathrm{prox}_{th}(z) - z}  \Big\} 
\end{equation}
provided that the mapping $\prox_{tf}$ is smooth around $2\mathrm{prox}_{th}(z)-z$.

\begin{proposition} \label{lemma:BD-regularity-drs}
 Fix any point $z \in \mathbb{R}^n$. Suppose that $h$ is convex, and that $\prox_{tf}$ is smooth and its Jacobian at $2\prox_{th}(z) - z$, denoted by $D$, satisfies $0\prec D\preceq I$. Then, the BD-regularity of $F_{\mathrm{DRS}}$ holds at $z$ if and only if $I-D$ is positive definite on the subspace $\mathrm{Ker}(I-M)$ for all $M \in \partial_B \mathrm{prox}_{th}(z)$.
\end{proposition}
\begin{proof}
% \proof{Proof.}
 $\Longleftarrow$. 
 Suppose the $BD$-regularity of $F_{\rm DRS}$ does not hold at $z$. Then there exists $0 \ne d \in \R^n$ such that $0\in \partial_{B} F_{\mathrm{DRS}}(z)d$, which implies that $0 \in \partial F_{\mathrm{DRS}}(z)d$. Together with \eqref{BJac-DRS}, there exists $M\in \partial\prox_{th}(z)$ such that 
 \begin{equation}\label{temp-equaM0}
  [M-D(2M-I)]d=0.
 \end{equation} 
 Let $e = (2M-I)d$. Then $e-Md = (M\!-\!I)d$ and equality \eqref{temp-equaM0} implies that $Md = De$. Hence,
\be \label{eq:drs-basic-2} (I\!-\!D)e=(M\!-\!I)d, \ee
and $e\ne 0$ (otherwise we will have $d = 0$. In fact, if $e =0$, then it holds $(M-I)d = 0$ and $Md = 0$, which gives $d=0$). Multiplying both sides of equation \eqref{eq:drs-basic-2} by $(Md)^\top$ leads to
\[ (Md)^\top(I\!-\!D)e+ d^\top M(I\!-\!M)d=0. \]
Together with $Md = De$, it follows that
$e^\top D(I\!-\!D)e + d^\top M(I\!-\!M)d=0$. From the given assumption on $D$, $e^\top D(I\!-\!D)e\ge 0$, while by the convexity of $h$ and \cite[Lemma 3.3.5]{milzarek2016numerical}, $d^\top M(I\!-\!M)d\ge 0$. 
Thus,
\[ e^\top D(I\!-\!D)e = d^\top M(M\!-\!I)d = 0. \]
Recall that $I\succeq D\succ 0$ and $I\succeq M\succeq 0$. So 
$D(I\!-\!D)e=0$ and $M(M\!-\!I)d = 0$. Since $D$ is assumed to be positive definite, the former is equivalent to $(I-D)e=0$. Thus,  
\[ (I\!-\!D)e = M(M\!-\!I)d = 0. \]
Using again the equation $Md = De$, we deduce that 
$De=e$ and $M(M\!-\!I)d = (M\!-\!I)De =Me - e = 0$, therefore $e = Me = De$. This means that $0 \ne e \in \mathrm{Ker}(I\!-\!M)$.
Recall that $M\in\partial{\rm prox}_{th}(z)$. There exist an integer $m\ge 1$, 
$\alpha_1\ge 0,\ldots,\alpha_m\ge 0$ with $\sum_{i=1}^m\alpha_i=1$, and 
$M_i\in\partial_{B}{\rm prox}_{th}(z)$ for $i=1,\ldots,m$ such that $M=\sum_{i=1}^m\alpha_iM_i$. Together with $e \in \mathrm{Ker}(I\!-\!M)$, it follows that 
\[
 \sum_{i=1}^m\alpha_i\langle e,(I-M_i)e\rangle=0. 
 \]
 Recall that $I-M_i\succeq 0$ for each $i\in\{1,\ldots,m\}$. We have $\alpha_i\langle e,(I-M_i)e\rangle=0$ for each $i\in\{1,\ldots,m\}$. 
Along with $\sum_{i=1}^m\alpha_i=1$, there necessarily exists an index $k\in\{1,2,\ldots,m\}$ such that $\alpha_k>0$ and $\langle e,(I-M_k)e\rangle=0$. This means that $(I-M_k)e=0$, i.e., $0\ne e\in{\rm Ker}(I-M_k)$. However, $(I-D)e=0$. Thus, we obtain a contradiction to the given assumption. The conclusion in this direction follows.

$\Longrightarrow$. 
Suppose that there exists $M\in\partial_B \mathrm{prox}_{th}(z)$ such that $I\!-\!D$ is not positive definite on  $\mathrm{Ker}(I\!-\!M)$. Then, there is $0\ne d\in\mathrm{Ker}(I\!-\!M)$ 
such that $(I\!-\!D)d = 0$. Consequently, 
\[
 [M\!-\!D(2M\!-\!I)]d = Md\!-\!Dd = d\!-\!Dd = 0,
 \]
 which yields a contradiction to the BD-regularity of $F_{\rm DRS}$ at $z$ by invoking the following inclusion
\begin{equation*}
 \partial_{B} F_{\mathrm{DRS}}(z)\supset\Big\{\big[M-D(2M-I)\big]\,:\, M \in \partial_{B}\mathrm{prox}_{th}(z), \, D = \nabla \mathrm{prox}_{tf}(y)|_{y = 2\mathrm{prox}_{th}(z) - z}  \Big\}. 
\end{equation*}
 Thus, the implication in this direction follows. The proof is completed.
\end{proof}

\begin{remark}
We note that the assumption on the positive definiteness of the Jacobian of $\mathrm{prox}_{tf}$ holds for any $C^2$ smooth and convex function $f$ with $L$-Lipschitz continuous gradient and $t < 1/L$.
 To the best of our knowledge, we are the first to provide an equivalent characterization of the BD-regularity of $F_{\rm DRS}$. This characterization can be effectively utilized to verify the BD-regularity employed in DRS-based SSN methods \cite{xiao2018regularized,li2018semismooth}.
\end{remark}
\begin{eg} \label{eg:lasso-drs}
Consider the case when $h(x) = \lambda\|x\|_1$ and $f$ is $C^2$ smooth and convex. For a root $z^*$ of $F_{\rm DRS}$, let $x^* = \prox_{th}(z^*)$ and define $T := \{ i \in [n] : x_i^* \ne 0 \} \cup \{ i \in [n]: x_i^* = 0, [|\nabla f(x^*)|]_i = \lambda \}$. It follows from the smoothness of $f$ and the definition of $\mathrm{prox}_{tf}$ that $z^* = x^* - t\nabla f(x^*)$. 
Furthermore, if $t < 1 /\lambda_{\max}(\nabla^2 f(x^*))$, we have
\[ \begin{aligned}
\nabla \prox_{tf}(y)|_{y = 2\prox_{th}(z^*) - z^*} & = \left( I + t\nabla^2 f((\mathrm{id} + t \nabla f)^{-1}(x^* + t \nabla f(x^*))) \right)^{-1} \\
& = \left( I + t\nabla^2 f(x^*) \right)^{-1}.
\end{aligned}
\]
From Proposition \ref{lemma:BD-regularity-drs} and Example \ref{eg}, the BD-regularity of $F_{\text{DRS}}$ holds at $z^*$ if and only if 
\[ d^\top \left[I - \left( I + t\nabla^2 f(x^*) \right)^{-1} \right]d > 0 \;\; \mathrm{for~all~  nonzero~} d \mathrm{~with}\, \supp\{d\} \text{$\subseteq$} T.  \]
Denote the eigenvalue decomposition of $ \nabla^2 f(x^*)$ by $\nabla^2 f(x^*) = U\Lambda U^\top$. We have
\be \label{eq:drs-4} d^\top (I - (I + t\nabla^2 f(x^*))^{-1}) d = d^\top U (I - (I + t \Lambda)^{-1}) U^\top d = t \sum_{i=1}^n \frac{\lambda_i}{1 + t\lambda_i} (u_i^\top d)^2 > 0.  \ee
Because of the fact $\lambda_i \geq 0, \; i=1, 2,\ldots, n$, inequality \eqref{eq:drs-4} holds if and only if $\lambda_j (u_j^\top d)^2 > 0$ for some $j\in\{1,2,\ldots,n\}$. This means that $d^\top \nabla^2 f(x^*) d > 0$. Therefore, the BD-regularity of $F_{\text{DRS}}$ holds at $z^*$ if and only if
$ [\nabla^2f(x^*)]_{TT} \succ 0$. 
\end{eg}

\subsection{Characterization for the gradient mapping}
As introduced in Section \ref{subsubsec:alm}, for \eqref{prob} with $C^2$ smooth and convex $f$ and convex $h$, given a fixed constant $t > 0$ and parameter $z \in \R^n$, the B-Jacobian of $F_{\rm ALM}(\cdot;z)$ is given by
\[ \partial_B F_{\mathrm{ALM}}(x; z) =\{\nabla^2\!f(x) + (I - M)/t: M \in \partial_B \prox_{t h}(x - t z)\}.  \]

\begin{proposition} \label{lemma:BD-regularity-alm}
 Fix a parameter $z\in \R^n$. For any point $x\in \mathbb{R}^n$, the BD-regularity of $F_{\mathrm{ALM}}(\cdot;z)$ holds at $x$ if and only if $\nabla^2 f(x)$ is positive definite on the subspace $\mathrm{Ker}(I-M)$ for all $M \in \partial_B \mathrm{prox}_{t h}(x - t z)$.
\end{proposition}
\begin{proof}
The proof follows the same arguments as in Proposition \ref{lemma:BD-regularity}.
\end{proof}

\begin{remark}
Note that $z^* = \nabla f(x^*)$ for any limiting point $(x^*, z^*)$ of the sequence $\{(x_k, z_k)\}$ generated by ALM \eqref{alg:alm}. When $z = \nabla f(x)$, it is easy to see by Proposition \ref{lemma:BD-regularity} and Proposition \ref{lemma:BD-regularity-alm} that the BD regularity of $F_{\mathrm{ALM}} (\,\cdot\, ; z)$ at point $x$ coincides with that of $F_{\rm PGM}$. Then by Remark \ref{rem:BD-PGM}, the BD-regularity of $F_{\rm ALM}(\,\cdot\, ; z)$, with $z = \nabla f(x^*)$,
holds at $x^*$  if the second-order sufficient condition and SC are satisfied at $(x^*, \nabla f(x^*))$ \cite[Theorem 5.4.4]{milzarek2016numerical}. We note that the specific condition for BD-regularity has also been analyzed in semidefinite programming \cite{zhao2010newton}.
\end{remark}

\subsection{Invertibility of the shifted B-Jacobian}
 Let $F$ be any residual mapping chosen from $F_{\rm PGM}$, $F_{\rm DRS}$, and $F_{\rm ALM}$. For the well-posedness of the SSN method \eqref{eq:ssn}, we need the following invertibility condition for the shifted B-Jacobian.
\begin{assum} \label{assum:invert-jacobian}
    There exists $b_2>0$ such that for all $x \in \bar{B}(x^*,b_2)$, it holds that \be \label{eq:eigen-J} \|(J(x) + \mu(x)I)^{-1}\| \leq \mu(x)^{-1}, \ee
where $\mu (x):= \|F(x)\|$ and $J(x)$ is any element of the B-Jacobian of $F$ at $x$.
\end{assum}

A sufficient condition for Assumption \ref{assum:invert-jacobian} is that ${\rm Re}(\lambda(J(x))) \subset (-\infty, -\delta] \cup [0, \infty)$ for all $x \in \bar{B}(x^*,b_2)$ and $J(x) \in \partial_B F(x)$, where $\delta > 0$ is a constant, $\lambda(A)$ is the set of eigenvalues of $A$, and ${\rm Re}(a)$ denotes the real part of a complex number $a$. In fact, from the Lipschitz continuity, we have $\mu(x) =\|F(x)\| \leq L\|x - x^*\| \leq Lb$ for all $ x \in \bar{B}(x^*,b)$. Without loss of generality, we assume $b \leq \frac{\delta}{2L}$. Then, it holds that
$ |{\rm Re}(\lambda(J(x) + \mu(x) I))| \geq \mu(x)$,  $\forall x \in \bar{B}(x^*,b)$,  
which gives \eqref{eq:eigen-J}. 
Note that the real parts of the eigenvalues of a monotone operator $F$ are always nonnegative since every element of its B-Jacobian is positive semidefinite on $\bar{B}(x^*, b)$ \cite[Proposition 2.1]{schaible1996generalized}.   We further note  \cite[Proposition 2.3]{xiao2018regularized} that the local monotonicity of $F_{{\rm PGM}}$ and $F_{{\rm DRS}}$ on $\bar{B}(x^*,b)$ holds if both $f$ and $h$ are  locally convex around $x^*$ and ${\rm prox}_{tf}(x^*)$, respectively. Thus, the local convexity of $f$ and $h$ is a sufficient condition for Assumption \ref{assum:invert-jacobian}. It is worth mentioning that the real parts of the eigenvalues of the B-Jacobian satisfying \eqref{eq:eigen-J} can be strictly negative, as long as they are bounded from above by $-\delta$, which allows $f$ or $h$ to be nonconvex on $\bar{B}(x^*,b)$.

For illustration of Assumption \ref{assum:invert-jacobian}, consider the cases when $f$ is $C^2$ smooth and $h(x) = \lambda \|x\|_1$. We first present the following lemma about a local constancy property of the Jacobian of ${\rm prox}_{th}$.
\begin{lemma} \label{lemma:const}
Consider \eqref{prob} with $C^2$ smooth $f$ and $h(x) = \lambda \|x\|_1$. For a root $x^*$ of $F_{{\rm PGM}}$, suppose that SC holds at $x^*$. Then, there exists a neighborhood $\bar{B}(x^*,b)$ such that the Jacobian of ${\rm prox}_{th}$ is constant on $\{x - t \nabla f(x): x \in \bar{B}(x^*,b)\}$.  
\end{lemma}
\begin{proof}
Note that SC implies $\{i\in [n]: x^*_i = 0, \; |[\nabla f(x^*)]_i| = \lambda \}$ is empty. Define $I_1(x^*) = \{i \in [n]: x^*_i \ne 0\}$ and $I_2(x^*) = \{i \in [n] : x^*_i =0, \; |[\nabla f(x^*)]_i| < \lambda \}$. Then, $I_1(x^*) \cup I_2(x^*) = [n]$. We therefore deduce that $x^*_i \ne 0$ if and only if $|x^*_i - t [\nabla f(x^*)]_i| > t\lambda$. 
We can equivalently rewrite $I_1(x^*)$ as $I_1(x^*) = \{ i \in [n]: |x^*_i - t [\nabla f(x^*)]_i| > t\lambda \}$. Similarly, $I_2(x^*)$ has the equivalent formulation $I_2(x^*) = \{i\in [n]: |x^*_i - t [\nabla f(x^*)]_i| < t\lambda \}.$ Hence, by the smoothness of $f$, there exists a neighborhood $\bar{B}(x^*,b)$ such that $I_1(x) = I_1(x^*) $ and $I_2(x) = I_2(x^*)$ for all $x \in \bar{B}(x^*,b)$. This  together with \eqref{eq:jaco-l1} gives the desired result.
\end{proof}

With the above lemma, we are able to relate Assumption \ref{assum:invert-jacobian} to the requirements on $\nabla^2 f(x)$. 
\begin{theorem} \label{thm:eigen}
 Consider \eqref{prob} with $C^2$ smooth $f$ and $h(x) = \lambda \|x\|_1$. For a root $x^*$ of $F_{{\rm PGM}}$, suppose that SC holds at $x^*$. Define $\revise{\Omega} = \{i \in [n]: x^*_i \ne 0\}$.
If the eigenvalues $\lambda([\nabla^2 f(x)]_{\revise{\Omega \Omega}}) \in (-\infty, -\delta] \cup [0, \infty), \; \forall x \in \bar{B}(x^*, b)$ for some constants $\delta, b >0$, then the Assumption \ref{assum:invert-jacobian} holds. 
\end{theorem}
\begin{proof}
By Lemma \ref{lemma:const}, we can assume without loss of generality that the Jacobian ${\rm prox}_{th}$ is constant on $\{x - t \nabla f(x): x \in \bar{B}(x^*,b)\}$. The Jacobian of $F_{{\rm PGM}}$ is 
\[ J(x) = \left(\begin{array}{cc}
    t [\nabla^2 f(x)]_{\revise{\Omega \Omega}} & t[\nabla^2 f(x)]_{\revise{\Omega \Omega^c}}  \\
    0 & I 
\end{array} \right), \quad \text{for all } x \in \bar{B}(x^*,b). \]
Its eigenvalues consists of the eigenvalues of $t [\nabla^2 f(x)]_{\revise{\Omega \Omega}}$ and $1$. Hence, if $\lambda([\nabla^2 f(x)]_{\revise{\Omega \Omega}})$ $\subset (-\infty, -\delta] \cup [0, \infty), \; \forall x \in \bar{B}(x^*, b)$, then $\lambda(J(x)) \subset (-\infty, -t\delta] \cup [0, \infty)$. Hence, for the choice of $\mu(x) =\|F(x)\|$, Assumption \ref{assum:invert-jacobian} holds.
\end{proof}

\begin{remark}
Theorem \ref{thm:eigen} presents a scenario where the natural residual $F_{{\rm PGM}}$ of a nonconvex problem \eqref{prob} satisfies Assumption \ref{assum:invert-jacobian}. Let us note that the requirement on $\lambda([\nabla^2 f(x)]_{\revise{\Omega \Omega}})$ is satisfied if $\nabla^2 f(x)$ is constant on $\bar{B}(x^*, b)$, which corresponds to the case where $f$ is a quadratic function.
\end{remark}

\section{Local smoothness} \label{sec:sc}

The existing proofs of the superlinear or quadratic convergence of the SSN method often require the BD-regularity condition \cite{qi1993convergence,pang1993nonsmooth}, which implies the isolatedness of the minimizer. It is a strong condition that may not hold in practice. For the cases where the solution set contains nonisolated points, there exist extensive literatures using error bound \cite{luo1993error,fan2004inexact,yue2019family}, Kurdyka-\L ojasiewicz property \cite{attouch2010proximal}, and Polyak-\L ojasiewicz property \cite{karimi2016linear}, among others, to establish  superlinear convergence rates of Newton-type methods. The analysis in most of the existing works relies on the mapping $F$ being locally smooth. In this section, we present sufficient conditions under which such local smoothness holds for the three systems \eqref{eq:pgm}, \eqref{eq:drs}, and \eqref{eq:gradient},  respectively. 

\subsection{Local smoothness of proximal mappings} \label{subsec:smooth-prox}
 To derive the local smoothness of the residual mappings in \eqref{eq:pgm}, \eqref{eq:drs}, and \eqref{eq:gradient}, 
 we first take a closer look at conditions to guarantee the local smoothness of proximal mappings.

\subsubsection{Partial smoothness}
The concept of partial smoothness \cite{lewis2002active, lewis14optimality, lewis16generic} is useful in establishing the smoothness of proximal mappings. In \cite{daniilidis06geometric}, the authors proved the smoothness of proximal mappings under an stronger version of the partial smoothness condition, which was introduced in \cite{lewis2002active}. Later on, Lewis et al. reintroduced partial smoothness in \cite{lewis14optimality} as the following weaker condition, and showed in \cite{lewis16generic} that it holds generically for all semialgebraic functions. 
%--------------------------------------------------------------------------

\vskip 2mm
\begin{definition}[$C^p$-partial smoothness]\label{def-psmooth}
 Consider a proper closed function $\phi\!:\R^n\rightarrow \bar{\mathbb{R}}$ and a $C^p$ $(p \ge 2)$ embedded submanifold $\Mcal$ of $\R^n$. The function $\phi$ is said to be $C^p$-partly smooth at $x \in \Mcal$ for $v \in \partial \phi(x)$ relative to $\Mcal$ if
    \begin{itemize}
        \item[(i)] Smoothness: $\phi$ restricted to $\mathcal{M}$ is $C^{p}$-smooth near $x$.
        \item[(ii)] Prox-regularity: $\phi$ is prox-regular at $x$ for $v$. 
        
        \item[(iii)] Sharpness: $\mathrm{par}\, \hat{\partial} \phi(x) = N_{\Mcal} (x)$, where $N_{\Mcal} (x)$ is the normal space to $\Mcal$ at $x$.
        \item[(iv)] Continuity: There exists a neighborhood $V$ of $v$ such that the set-valued mapping $V \cap \partial \phi$ is inner semicontinuous at $x$ relative to $\mathcal{M}.$
    \end{itemize}
\end{definition}

\vskip 2mm
\begin{remark}
 Note that the sharpness condition (iii) is sometimes written as $\mathrm{par}\, \partial_P \phi(x) = N_{\Mcal} (x)$, where $\partial_P \phi(x)$ denotes the set of proximal subgradients of $\phi$ at point $x$. When $\phi$ is prox-regular at $x$ for $v$, by definition there is a neighborhood $V$ of $v$ such that $\hat{\partial}\phi (x)\cap V=\partial_{p} \phi (x)\cap V$, which along with the convexity of the sets $\hat{\partial} \phi(x)$ and $\partial_p\phi(x)$ implies that $\mathrm{par}\, \hat{\partial} \phi(x) = \mathrm{par}\, \partial_p\phi(x)$. For a reference on proximal subgradients, see \cite{rockafellar2009variational}.
\end{remark}
\vskip 2mm

Drusvyatskiy et al. \cite{lewis14optimality} showed that for a set $S \subset \R^n$, assuming the $C^2$-partial smoothness of the indicator function $\delta_S$ at a point $x^* \in S$ for a normal vector $v^*$ plus a relative interior point condition, the projection mapping onto $S$ is smooth around $x^* + t v^*$, where $t$ is any sufficiently small positive constant. The authors also note that this result may be extended to the proximal mappings of partly smooth functions. We provide a proof below for such an extension, following a similar scheme as in \cite{lewis08projection, lewis14optimality}. Similar to \cite[Proposition 9.7]{lewis14optimality}, we also require an extra relative interior condition: $v^* \in \mathrm{ri}\, \hat \partial \phi (x^*)$.
%--------------------------------------------------------------------
 \begin{lemma} \label{lemma:prox-smooth}
  Consider a proper closed function $\phi\!:\mathbb{R}^n \rightarrow \bar{\mathbb{R}}$ that is prox-bounded and $C^p$-partly smooth $(p \ge 2)$ at $x^*$ for $v^*$ relative to $\mathcal{M}$. Suppose $v^*\in \mathrm{ri}\, \hat{\partial} \phi(x^*)$. Then, for all sufficiently small $t > 0$, the proximal mapping $\mathrm{prox}_{t \phi}$ is $C^{p-1}$-smooth near $x^*+t v^*$.
\end{lemma}
\begin{proof}
  By Definition \ref{def-psmooth} (ii) and Theorem \ref{prop:lip-prox-regular}, there exists $t_0 > 0$ such that for all $t \in (0, t_0)$, the mapping $\mathrm{prox}_{t \phi}$ is single-valued and Lipschitz continuous around $x^*+t v^*$. Fix any $t\in(0,t_0)$ and let $w^*:= x^*+t v^*$. We claim that $\mathrm{prox}_{t \phi} (w) \in \mathcal{M}$ for all $w$ around $w^*$. If not, there exists a sequence $w_k \rightarrow w^*$ such that $x_k := \mathrm{prox}_{t \phi} (w_k) \notin \mathcal{M}$ for all $k$. Note that $\mathrm{prox}_{t\phi} = (\mathrm{id} + t T)^{-1}$ for a $\phi$-attentive localization $T$ of $\partial \phi$. Hence $v_k := \frac{w_k - x_k}{t} \in \partial \phi (x_k)$. Moreover, we have $\lim_{k\to\infty} x_k = x^*$ by the continuity of $\mathrm{prox}_{t \phi}$. Therefore $\lim_{k \rightarrow \infty} v_k = \frac{w^* - x^*}{t} = v^*$. In addition, from the continuity of the Moreau envelope, $\lim_{k \rightarrow \infty} \phi(x_k) = \phi(x^*)$. By invoking  \cite[Proposition 10.12]{lewis14optimality}, we conclude that $x_k \in \mathcal{M}$ for all large enough $k$, which is a contradiction to $x_k\notin \mathcal{M}$ for all $k$.
    
  By Definition \ref{def-psmooth} (i) and (iii), there exists a $C^p$-smooth function $\widetilde{\phi}: U \rightarrow \mathbb{R}$, where $U$ is a neighborhood of $x^*$ in $\mathbb{R}^n$, such that $\widetilde{\phi}|_{\mathcal{M} \cap U} = \phi|_{\mathcal{M} \cap U}$ and 
  \begin{equation}\label{eq:phi-nm}
      \nabla \widetilde{\phi}(x^*) - v^* \in N_{\mathcal{M}} (x^*).
  \end{equation}
  Since $\mathcal{M}$ is a $C^p$-smooth manifold, there is an open set $U_1 \subseteq U$ containing $x^*$ and a $C^p$-smooth map $H\!:U_1 \rightarrow \mathbb{R}^m$ such that
  $\mathcal{M} \cap U_1 = \{x \in U_1\ |\ H(x) = 0\}$
  and $\nabla H(x)$ is surjective for all $x \in U_1$. We then have $N_{\mathcal{M}} (x) = \mathrm{Range}( \nabla H(x)^{\top})$ for all $x \in U_1$. By the inclusion \eqref{eq:phi-nm}, there exists $r^* \in \mathbb{R}^m$ such that $\nabla \widetilde{\phi}(x^*) - v^* + \nabla H(x^*)^{\top} r^*=0$. Since $\mathrm{prox}_{t \phi} (w) \in \mathcal{M}$ for all $w$ around $w^*$, there is a neighborhood $V$ of $w^*$ such that 
  \begin{equation}\label{def-proxphi}
   \mathrm{prox}_{t \phi} (w) = \mathrm{argmin}_{x \in \mathcal{M} \cap U} \left\{ \tilde{\phi}(x) + \frac{1}{2t} \|x - w\|^2 \right\} \quad\ {\rm for\ all}\ w \in V.
  \end{equation}
  Next we follow a similar scheme as in \cite[Lemma 2.1]{lewis08projection} to prove the desired result. Together with \eqref{def-proxphi}, there exists an open set $V_1 \subset V$ containing $w^*$ such that for each $w \in V_1$, $x = \mathrm{prox}_{t \phi} (w)$ if and only if
    \begin{equation}\label{eq:uwr}
        x \in U_1 \quad \text{and} \quad
        \begin{cases}
            w = t \nabla \tilde{\phi} (x) + x + \nabla H(x)^{\top} r \\
            0 = H(x)
        \end{cases}
        \text{for some } r\in \mathbb{R}^m.
    \end{equation}
  Define a $C^{p-1}$-smooth function $G: U_1 \times \mathbb{R}^m \rightarrow \mathbb{R}^n \times \mathbb{R}^m$ by
  \begin{equation*}
   G(x, r) := (t \nabla \widetilde{\phi} (x) + x + \nabla H(x)^{\top} r, H(x)).  
 \end{equation*}
 Recall that $\nabla \widetilde{\phi}(x^*) - v^* + \nabla H(x^*)^{\top} r^*=0$. It is immediate to have that 
 \begin{equation*}
  G(x^*, t r^*) = (t \nabla \widetilde{\phi} (x^*) + x^* + t \nabla H(x^*)^{\top} r^*, 0) = (w^*, 0).
 \end{equation*}
 Also, the linear operator $\nabla G(x^*, t r^*): \mathbb{R}^n \times \mathbb{R}^m \rightarrow \mathbb{R}^n \times \mathbb{R}^m$ takes the following form
 \begin{equation*}
  \nabla G(x^*, t r^*) (x, r) = (x + t A x + \nabla H(x^*)^{\top} r, \nabla H(x^*) x),
 \end{equation*}
 where $A\!:=\!\nabla^2 \widetilde{\phi} (x^*) +\!\sum_{i=1}^m r_i^* \nabla^2 H_i (x^*)$ and $H_i\!:\mathbb{R}^n\to\mathbb{R}$ is the $i$th component of $H$.
    
 Let $t_1 = \min\{t_0, - \frac{1}{\lambda_{\min} (A)}\}$ if $\lambda_{\min} (A) < 0$ and $t_1 = t_0$ otherwise. Fix $t \in (0, t_1)$. Then, $\nabla G(x^*, t r^*)$ is invertible. Indeed, for any $(x, q) \in \mathrm{Ker}( \nabla G(x^*, t r^*))$, we have $x \in T_{\mathcal{M}} (x^*)$ and $t x^{\top} A x + \|x\|^2 = 0$, which implies $x = 0$ and therefore $q= 0$ by the surjectivity of $\nabla H(x^*)$. By the inverse function theorem, there are open sets $S \subset U_1 \times \mathbb{R}^m$ containing $(x^*, t r^*)$ and $W \subset \mathbb{R}^n \times \mathbb{R}^m$ containing $(w^*, 0)$ such that the map $G\!: S \rightarrow W$ has a $C^{p-1}$ smooth inverse $G^{-1}\!: W \rightarrow S$.
  Let $V_2 = \{w \in V_1\ |\, (w, 0) \in W\}$. Then $V_2$ is a neighborhood of $w^*$. Fix any $w \in V_2$. Let $(x, r) = G^{-1} (w, 0)$. We have
  $ x \in U_1$, and $G(x, r) = (w, 0)$, which, by \eqref{eq:uwr}, means that $x = \mathrm{prox}_{t \phi}(w)$. Hence $\mathrm{prox}_{t \phi} = P \circ G^{-1} \circ P^*$ is $C^{p-1}$ smooth on $V_2$, where $P: \mathbb{R}^n \times \mathbb{R}^m \rightarrow \mathbb{R}^n$ is the canonical projection $(x, r) \mapsto x$ and $P^*: \mathbb{R}^n \mapsto \mathbb{R}^n \times \mathbb{R}^m$ is the embedding $x \mapsto (x, 0)$.
\end{proof}

\subsubsection{Closedness of the set of nondifferentiable points}
Next, we introduce another set of conditions that can also ensure the local smoothness of proximal mappings.
Let us start with the concept of twice epi-differentiability \cite[Definition 13.6]{rockafellar2009variational}. 
For a function $\phi: \R^n \to \bar{\R}$, a point $x \in \mathrm{dom}\, \phi$, and vectors $v, w \in \R^n$, we define the second-order quotient $\Delta_{t}^2\phi(x|v)[w]: = \frac{\phi(x+ tv) - \phi(x) -  t\iprod{v}{w}}{t^2 / 2}$ for all $t > 0$. 
\begin{definition}[twice epi-differentiability]
 A function $\phi$ is said to be twice epi-differentiable at $x$ for $v$ if $\Delta_{t}^2\phi(x|v)[w]$ epi-converges to ${\rm d}^2 \phi(x|v)[w]$ as $t\downarrow 0$, where ${\rm d}^2 \phi(x|v)[w]$ denotes the second subderivative of $\phi$ at $x$ for $v$, defined as $\liminf_{t\downarrow 0\atop w'\to w}\Delta_{t}^2\phi(x|v)[w']$.  
\end{definition}

Twice epi-differentiablity is a mild condition satisfied by all fully amenable functions \cite[Corollary 13.15]{rockafellar2009variational} and decomposable functions \cite{shapiro2003class,milzarek2016numerical}, which in particular include $\ell_1$ norm, group sparisty regularizer, nuclear norm, the indicator function of a polyhedral set, etc.  
To derive the differentiability of proximal mappings, another concept that we need is the generalized quadratic property of the second subderivative. 
\begin{definition}[Generalized quadratic second subderivative]
Let $\phi$ be twice epi-differentiable at $x$ for $v \in \partial \phi(x)$. We say that the second subderivative is generalized quadratic if 
\be \label{eq:gen-quad} {\rm d}^2 \phi(x|v)[w] = \iprod{w}{Mw} + \delta_{S}(w), \quad \forall\, w \in \R^n, \ee
where $S \subset \R^n$ is a linear subspace and $M \in \R^{n\times n}$.
\end{definition}

One can easily verify that any $C^2$ function satisfies the above definition. It has been shown in \cite[Theorem 4.5]{rockafellar1988first} and \cite[Lemma 5.3.27]{milzarek2016numerical} that $C^2$ fully decomposable functions \cite{shapiro2003class} have generalized quadratic second subderivative if $v \in {\rm ri}\, \hat \partial \phi(x)$, where the relative interior condition is needed to guarantee the existence of the subspace $S$. The $C^2$ fully decomposable functions include the $\ell_1$ norm, Ky Fan $k$-norm, and the indicator function of the positive semidefinite cone, see, e.g. \cite[Section 5]{milzarek2016numerical}. The generalized quadratic condition has been used to derive the differentiability of proximal mappings in \cite{poliquin1996generalized,rockafellar1988first,stella2017forward,themelis2018forward}.

When the set of nondifferentiable points of the proximal mapping is closed, we can deduce the local smoothness by restricting the second subderivative of the funtion to be generalized quadratic.
\begin{lemma} \label{lemma:continuous-diff}
Suppose that a prox-bounded function $\phi:\R^n \rightarrow \bar{\R}$ is both prox-regular and twice epi-differentiable at $x^*$ for $v^* \in \partial \phi(x^*)$ with its second subderivative being generalized quadratic. Let $\rho$ and $\epsilon$ be the constants in Definition \ref{def:prox-regular} of prox-regularity at $x^*$ for $v^*$. Assume that $0 < t  < 1 /\rho$ and that the set of nondifferentiable points of $\mathrm{prox}_{t \phi}$, denoted by $\mathcal{N}$, is closed. If the mapping $\mathrm{prox}_{t \phi}$ is $C^{p-1}$ on $\mathcal{N}^c$, then it is $C^{p-1}$ around $x^* + tv^*$.
\end{lemma} 
\begin{proof}
Set $\bar{\phi}(x):= \phi(x) - \iprod{v^*}{x - x^*} + (R/2)\|x - x^*\|^2$ with $R > \rho$ being sufficiently large. Then, we can deduce from \cite[Proposition 13.37]{rockafellar2009variational} that $x^* = \argmin\; \bar{\phi}$. By \cite[Theorem 3.9]{poliquin1996generalized}, $\mathrm{prox}_{t \phi}$ is differentiable at $ x^* + tv^*$. From the $C^{p-1}$ smoothness of $\mathrm{prox}_{t \phi}$ over $\mathcal{N}^c$, we conclude that $\mathrm{prox}_{t \phi}$ is $C^{p-1}$ around $x^* + tv^*$.
\end{proof}

\subsubsection{Examples of nonsmooth functions with locally smooth proximal mappings} \label{subsec:exam}
Let us show some examples satisfying the partial smoothness condition. Meanwhile, we also investigate the nondifferentiable points of the proximal mappings.

\vskip 2mm
\begin{eg}[functions of vectors]
When $h$ is a certain vector norm or the indicator function of some simple set, it is partly smooth, and the set of nondifferentiable points of its proximal mapping is closed.  
Consider the following functions on space $\mathbb{R}^n$ around any point $x^* \in \mathbb{R}^n$.
\begin{itemize}
    \item 
    $h(x) = \|x\|_1$.
    We have for all $x \in \mathbb{R}^n$,
    \begin{equation*}
        \partial h(x) = \{v \in \mathbb{R}^n : v_i \in [-1, 1], v_j = \mathrm{sign}(x_j), \forall\, i \in {\rm supp}(x)^c, j \in {\rm supp}(x)\},
    \end{equation*}
    where ${\rm supp}(x)^c =\{i \in [n]: x_i =0 \}$. Let $\mathcal{M}_{x^*} = \{x : x_i = 0, \forall\, i \in {\rm supp}(x^*)^c\}$. For all $x \in \mathcal{M}_{x^*}$ sufficiently close to $x^*$, we have ${\rm supp}(x) = {\rm supp}(x^*)$. It is therefore easy to verify that $h$ is partly smooth at $x^*$ relative to $\mathcal{M}_{x^*}$. From the expression of the proximal mapping of $\|x\|_1$, we deduce that the set of nondifferentiable points of $\mathrm{prox}_h$ is $\{ x: \exists\, i ~ \text{s.t.} \; |x_i| = 1 \}$, which is closed.
    
    \item
    $h(x) = \|x\|_p$ with $p \geq 2$. When $x^* \neq 0$, $h$ is $C^2$ near $x^*$, hence partly smooth there relative to $\mathcal{M}_{x^*} = \mathbb{R}^n$. When $x^* = 0$, $\mathrm{par}\, \partial h (x^*) = \mathbb{R}^n$, so $h$ is partly smooth there relative to $\mathcal{M}_{x^*} = \{x^*\}$. Its proximal operator is 
    \[ \mathrm{prox}_h(x) = \begin{cases}
        & \left(1 - \frac{1}{\|x\|_q} \right)x, \;\; \|x\|_q \geq 1, \\
        & 0, \;\; \mathrm{otherwise,}
    \end{cases} \]
    where $\|x\|_q$ is the dual norm with $\frac{1}{p} + \frac{1}{q} = 1$.
    The set of nondifferentiable points is $\{x:\|x\|_q = 1\}$, which is closed.
    
    \item
    $h(x) = \delta_{\{x : x \geq 0\}}(x)$. 
    % Define $I(x) = \{i: x_i = 0\}$.
    Since the set $\Mcal_{x^*} := \{x \in \R^n: x_i = 0, \; \forall i \in {\rm supp}(x^*)^c\}$ is a smooth manifold, $h$ is partly smooth at $x^*$ relative to $\Mcal_{x^*}$ \cite[Example 3.2]{lewis2002active}. Its proximal mapping is $\mathrm{prox}_{h}(x) = \max(x,0)$. The set of nondifferentiable points of $\mathrm{prox}_{h}$ is $\{ x: \exists\, i  ~\text{s.t.} \; x_i = 0\}$, which is closed. 
    
    \item
    $h(x) = \delta_{\{x: Ax = b\}}(x)$. It is obvious that $h$ is partly smooth relative to $\mathcal{M}_{x^*} = \{x: Ax = b\}$ near $x^*$ with $Ax^* = b$. Let $A^\dag$ be the Moore-Penrose pseudoinverse of $A$. Then, 
    $ \mathrm{prox}_h(x) = x - A^\dag (Ax - b)$, 
    which is everywhere differentiable.
    \end{itemize}
\end{eg}

\vskip 2mm
\begin{eg}[functions of matrices]
    When $h$ is a certain matrix norm or the indicator function of some simple set, it is partly smooth, and the set of nondifferentiable points of its proximal mapping is closed. 
    Consider the following functions on space $\mathbb{R}^{m \times n}$ around any matrix $X^* \in \mathbb{R}^{m \times n}$.

    \begin{itemize}
        \item 
        $h(X) = \|X\|_{2, 1}$.  
        Define $\revise{\Omega}(X) = \{i : X_i = 0\}, \forall\, X = (X_1, \dots, X_n) \in \mathbb{R}^{m \times n}$. Using the separability of partial smoothness \cite{lewis2002active}, we know that $h$ is partly smooth at $X^*$ relative to $\mathcal{M}_{X^*} = \{X \in \mathbb{R}^{m \times n} : \revise{\Omega}(X) = \revise{\Omega}(X^*) \}$. 
        Its proximal mapping is 
        $\mathrm{prox}_h(X_i) = \left(1 - \frac{1}{\max\{
        \|X_i\|_2, 1\}} \right)X_i$, 
        where $X_i$ is the $i$-th column of $X$. The set of nondifferentiable points is $\{X: \exists\, i ~ \text{s.t.} \; \|X_i\| = 1 \}$, which is closed.
    
        \item 
        $h(X) = \|X\|_*$, where $\|\cdot\|_*$ denotes the nuclear norm. It is shown in \cite[Example 4.4]{liang2014local} that $h$ is partly smooth at $X^*$ relative to $\Mcal_{X^*}:=\{ X \in \R^{m\times n}: \mathrm{rank}(X) = \mathrm{rank}(X^*) \}$. By the expression of the B-Jacobian of $\mathrm{prox}_{h}$ given in \cite[15.6.3e]{patrinos2014forward}, the set of nondifferentiable points are $\{X: X \mathrm{~has~singular~value~} 1\}$, which is closed by the continuity of the singular value function. 
        
    \end{itemize}
\end{eg}

\vskip 2mm
\begin{eg}
Consider the function $h(X) = \delta_{\{X: X \succeq 0\}}(X)$ defined on the Euclidean space $\mathbf{S}^n$ consisting of all $n$-by-$n$ real symmetric matrices, and any matrix $X^* \in \mathbf{S}^n$. It follows from \cite[Example 4.14 and Theorem 4.2]{lewis2002active} that $h$ is partly smooth at $X^*$ relative to $\Mcal_{X^*} :=\{ X \in \mathbf{S}^n :  X \succeq 0,\; \mathrm{rank}(X) = \mathrm{rank}(X^*)\}$. It follows from \cite[15.6.2h]{patrinos2014forward} that the set of nondifferentiable points is $\{X\succeq 0: X \mathrm{~has~eigenvalue~} 0\}$, which is closed.
\end{eg}

\subsection{Strict complementarity} \label{subsec:SC}
Recall that in Lemma \ref{lemma:prox-smooth} we required a relative interior condition. In our composite optimization setting, this condition is more widely known as strict complementarity, to which we now give a brief introduction.

Denote by $X^*$ the solution set of \eqref{prob:nonsmootheq} and pick any $x^* \in X^*$. %Consider a solution $x^*$ of problem \eqref{prob}. 
We first make the following technical assumptions on $f$ and $h$:
\begin{assum} \label{assum:basic-setting}
\begin{enumerate}
    \item Both $f$ and $h$ are prox-regular at $x^*$.
    \item The only $(v_1, v_2) \in \partial^{\infty} f(x^*) \times \partial^{\infty} h(x^*)$ with $v_1 + v_2 = 0$ is $(v_1, v_2) = (0, 0)$.
\end{enumerate}
\end{assum}
Under the above conditions, we have $\partial f(x^*) = \hat \partial f(x^*)$, $\partial h(x^*) = \hat \partial h(x^*)$, and
\begin{align}
    \partial (f+h)(x^*) &= \partial f(x^*) + \partial h(x^*), \nonumber \\ \mathrm{ri} \, \partial (f+h)(x^*) &= \mathrm{ri}\, \partial f(x^*) + \mathrm{ri} \, \partial h(x^*). \label{eq:ri-sum}
\end{align}

\begin{definition}[Strict complementarity] \label{def:sc}
We say that the strict complementarity condition (SC) holds at $x^*$ if 
    $0 \in \mathrm{ri}\, \partial (f+h)(x^*)$.
\end{definition}

\begin{remark}
By \eqref{eq:ri-sum}, SC is equivalent to the existence of a vector $z^* \in \mathbb{R}^n$ such that
\be \label{eq:sc} \frac{x^* - z^*}{t} \in \mathrm{ri}\, \partial f(x^*), \quad \frac{z^* - x^*}{t} \in \mathrm{ri}\, \partial h(x^*), \ee
where $t > 0$. 
\end{remark}

\begin{remark}
When both $f$ and $h$ are convex functions, a sufficient condition for Assumption \ref{assum:basic-setting} is $0 \in \mathrm{ri}\left(\mathrm{dom}\; f - \mathrm{dom}\; h\right)$.
\end{remark}

Let us take the $\ell_1$ regularized composite optimization problem and the basis pursuit problem as examples to explain SC defined in Definition \ref{def:sc}.
\begin{eg}[$\ell_1$ regularized composite optimization]
For $\psi(x) = f(x) + \lambda \|x\|_1$ where $f$ is a $C^2$ smooth function, by the expression of the subgradient of $\ell_1$ norm, SC holds at $x$ if 
$ \{i\in [n]: x_i = 0, [|\nabla f(x)|]_i = \lambda \} = \emptyset$, 
namely for each $i$ with $x_i = 0$, $[|\nabla f(x)|]_i \ne \lambda.$ 
\end{eg}
\begin{eg}[The basis pursuit problem] \label{eq:bp-sc}
When $f = \delta_{\{x: Ax = b\}}$ and $h = \|\cdot\|_1$ for some $A \in \R^{m\times n}$ and $b \in \R^m$, problem \eqref{prob} reduces to the basis pursuit problem:
\begin{equation} \label{bp-primal1}\min_{x \in \R^n} \; \|x\|_1, \;\; \st \;\; Ax = b. \end{equation}
Its dual is
\begin{equation} \label{bp-dual1} \min_{y\in \R^m} \; -b^\top y, \;\; \st \;\; \|A^\top y\|_{\infty} \leq 1.  \end{equation}
Let $(x^*, y^*)$ be a solution pair. Fix any $t > 0$ and let $z^* = x^* - t A^{\top} y^*$. Then $(x^*, z^*)$ satisfies \eqref{eq:sc}. Since
   $ \mathrm{ri}\left(\partial f(x^*)\right) = \partial f(x^*) = \mathrm{Range} (A^{\top})$,
the first half of \eqref{eq:sc} is always satisfied. Define $\revise{\Omega}(x^*) = \{ i : x_i^* = 0 \}$ and let $\revise{\Omega}(x^*)^c$ be its complement. The SC condition boils down to
\begin{equation*}
    - A^{\top} y^* = \frac{z^* - x^*}{t} \in \{v : v_i \in (-1, 1), v_j = \sign(x_j), \forall\, i \in \revise{\Omega}(x^*), j \in \revise{\Omega}(x^*)^c\},
\end{equation*}
namely for each $i$, either $1 - |A^\top y^*|_i$ or $x_i^*$ is zero but not both.
\end{eg}

\subsection{Local smoothness of the residual mappings}

Based on the results given in Lemmas \ref{lemma:prox-smooth} and \ref{lemma:continuous-diff}, we provide sufficient conditions for the local smoothness of the residual mappings $F_{{\rm PGM}}$, $F_{{\rm DRS}}$, and $F_{\rm ALM}$. The idea is to combine the smoothness results of the proximal mapping in Section \ref{subsec:smooth-prox} and the strict complementarity in Section \ref{subsec:SC}.

\subsubsection{The natural residual $F_{\rm PGM}$}

When considering the natural residual, we assume the function $f$ is smooth. Let us start with the following condition. 
\begin{cond} \label{cond:cond-pgm}
For $F_{{\rm PGM}}$, we give the following conditions.
\begin{itemize}
    \item[{\rm (B1)}] Assume that SC holds at $x^*$. The function $f$ is smooth. In addition, there exists a smooth manifold $\Mcal^h_{x^*}$ such that $h$ is $C^p$-partly smooth at $x^* \in \Mcal^h_{x^*}$ for $- \nabla f(x^*)$.  
    \item[{\rm (B2)}] The function $f$ is smooth and $h$ is twice epi-differentiable at $x^*$ for $-\nabla f(x^*)$ with second subderivative being generalized quadratic.
    The set of nondifferentiable points of  $\mathrm{prox}_{th}$ is closed. In addition, $\mathrm{prox}_{th}$ is $C^{p-1}$ in the complement of the set of their nondifferentiable points.
\end{itemize}
\end{cond}

\begin{remark}
As noted in Section \ref{lemma:continuous-diff}, any $C^2$-fully decomposable function has generalized quadratic second subderivative if the SC condition holds. It has also been shown that a $C^2$-fully decomposable function is partly smooth \cite{shapiro2003class}. However, there is no direct implication between the partial smoothness and the second subderivative being generalized quadratic. 
\end{remark}

Based on Condition \ref{cond:cond-pgm} and Lemmas \ref{lemma:prox-smooth} and \ref{lemma:continuous-diff}, we have the following corollary on the local smoothness of the natural residual \eqref{eq:pgm}.
\begin{theorem} \label{coro-pgm} 
If {\rm (B1)} or {\rm (B2)} holds, $F_{{\rm PGM}}$ is locally $C^{p-1}$ around $x^*$.
\end{theorem}
\begin{proof}
If (B1) or (B2) holds, it is clear from Lemmas \ref{lemma:prox-smooth} or \ref{lemma:continuous-diff} that $\mathrm{prox}_{th}$ is locally $C^{p-1}$ around $x^* - t\nabla f(x^*)$. Hence, the natural residual is local $C^{p-1}$ around $x^*$.
\end{proof}

\subsubsection{The DRS residual $F_{\rm DRS}$}

Unlike $F_{\rm PGM}$, $F_{\rm DRS}$ can also  deal with nonsmooth $f$.
\begin{cond} \label{cond:cond-drs}
For $F_{{\rm DRS}}$, we give the following conditions.
\begin{itemize}
\item[{\rm (B1$^\prime$)}] Assume that SC holds at $x^*$. Let $z^*$ be the corresponding vector such that \eqref{eq:sc} holds. There exist two smooth manifolds,  $\Mcal^f_{x^*}$ and $\Mcal^h_{x^*}$ such that
    the function $f$ is $C^p$-partly smooth at $x^* \in \Mcal^f_{x^*}$ for $\frac{x^* - z^*}{t}$ and $h$ is $C^p$-partly smooth at $x^* \in \Mcal^h_{x^*}$ for $\frac{z^* - x^*}{t}$.  
    \item[{\rm (B2$^\prime$)}] The function $f$ is twice epi-differentiable at $x^*$ for $\frac{x^* - z^*}{t}$ with second subderivative being generalized quadratic and $h$ is twice epi-differentiable at $x^*$ for $\frac{z^* - x^*}{t}$ with second subderivative being generalized quadratic. The set of nondifferentiable points of $\mathrm{prox}_{tf}$ and $\mathrm{prox}_{th}$ are both closed. In addition, both $\mathrm{prox}_{tf}$ and $\mathrm{prox}_{th}$ are $C^{p-1}$ in the complement of the set of their nondifferentiable points.
\end{itemize}
\end{cond}

Based on Condition \ref{cond:cond-drs} and Lemmas \ref{lemma:prox-smooth} and \ref{lemma:continuous-diff}, we have the following corollaries on the local smoothness of the DRS residual \eqref{eq:drs}.
\begin{theorem} \label{coro-drs}
If {\rm (B1$^\prime$)} or {\rm (B2$^\prime$)} holds, $F_{{\rm DRS}}$ is locally $C^{p-1}$ around $z^*$.
\end{theorem}
\begin{proof}
Suppose that (B1$^\prime$) or (B2$^\prime$) holds. Applying the results of Lemmas \ref{lemma:prox-smooth} or \ref{lemma:continuous-diff} implies that the proximal mappings $\mathrm{prox}_{tf}$ and $\mathrm{prox}_{th}$ are $C^{p-1}$ smooth around $2x^* - z^*$ and $z^*$, respectively. Therefore, the DRS residual is locally $C^{p-1}$ around $z^*$. 
\end{proof}

\subsubsection{The gradient mapping $F_{\mathrm{ALM}}$}

We require that $f$ be convex and smooth and $h$ be convex.
Based on Lemmas \ref{lemma:prox-smooth} and \ref{lemma:continuous-diff}, we have the following result on the local smoothness of the gradient mapping \eqref{eq:gradient}.

\begin{theorem} \label{coro-alm}
Consider the gradient mapping \eqref{eq:gradient} with convex and smooth $f$ and convex $h$.
If {\rm (B1)} or {\rm (B2)} holds, then for any fixed $t > 0$ sufficiently small and $z$ near $\nabla f(x^*)$, the mapping $F_{\mathrm{ALM}} (\cdot ; z)$ is locally $C^{p-1}$ around $x^*$.
\end{theorem}
\begin{proof}
If (B1) or (B2) holds, it is clear from Lemmas \ref{lemma:prox-smooth} or \ref{lemma:continuous-diff} that for a sufficiently small $t > 0$, the mapping $\mathrm{prox}_{th}$ is locally $C^{p-1}$ around $x^* - t\nabla f(x^*)$. 
Then by taking $z$ sufficiently close to $\nabla f(x^*)$, the gradient mapping $F_{\mathrm{ALM}} (\cdot ; z)$ is locally $C^{p-1}$ around $x^*$.
\end{proof}

\section{Convergence rate analysis of a projected SSN method} \label{sec:convergence}

Our analysis in Section \ref{sec:sc} shows that the residual mappings $F$ constructed by PGM, DRS and ALM are indeed  
locally smooth rather than semismooth if certain conditions, e.g., partial smoothness plus SC, are satisfied. It then implies that the SSN method reduces to the Newton method locally. In this section, we assume that there exists a smooth submanifold $\Mcal$ around a point $x^* \in X^*$ and a neighborhood $\bar{B}(x^*, b_1)$ with $b_1 > 0$ such that $F$ is smooth in $\Mcal \cap \bar{B}(x^*, b_1)$ but is still semismooth in $\bar{B}(x^*, b_1)$. In general, such $\Mcal$ is unknown as it relates to the optimal solution set. 
 In the cases considered in Section \ref{sec:sc}, the manifold $\Mcal$ is simply the full space $\R^n$. More generally, such a manifold may exist and be identified by certain first-order methods, if each component of $F$ is partly smooth and an appropriate version of the SC condition holds \cite{lewis2002active,liang2017activity,bareilles2022newton}. 
 It would be a separate interesting topic to explore the identification of $\Mcal$ without the SC condition.

\begin{eg}
Consider problem \eqref{prob} 
where $f$ is $C^2$ smooth and $h(x) = \lambda \|x\|_1$.
Suppose that $\Mcal=\{x \in \R^n : x_i = 0 {\rm ~if~} x_i^* = 0 \}$ is known where $x^*$ is an optimal solution. In addition, assume that SC does not hold at $x^*$, i.e., there exists some $i_0$ with $x_{i_0}^* = 0$ such that $[|\nabla f(x^*)|]_i = \lambda$. Note that $x^* = {\rm prox}_{th}(x^* - t \nabla f(x^*))$ for  $0 < t \leq 1/L_f $ with $L_f$ being the Lipschitz constant of $\nabla f$. Denote $z = x^* - t \nabla f(x^*)$. It then holds that $z_i > t$ or $z_i < -t$ for any $i \in {\rm supp}(x^*)$. 
Hence, for any $y \in \Mcal$ sufficiently close to $x^*$, we have for any $i \in {\rm supp}(x^*)$,
$(y - t \nabla f(y))_i > t$ if $z_i > t$ and $(y - t \nabla f(y))_i < -t$ if $z_i < -t$. Since ${\rm prox}_{th}$ is separable and each component function is smooth over $(t, +\infty) \cup (-\infty, -t)$, $F_{\rm PGM}$ is local smooth around $x^*$ on $\Mcal$, but not smooth over the neighborhood of $x^*$ in $\R^n$ (because the $i_0$-th component function of ${\rm prox}_{th}$ is nonsmooth around $x^*$). Note that if the SC holds at $x^*$, the domain of the local smoothness of $F_{\rm PGM}$ can be improved from $\Mcal$ to $\R^n$. 
\end{eg}

To utilize the local smoothness of $F$ over $\Mcal$, we need to restrict the iterates to $\Mcal$. Then, we consider solving the following system
\be \label{eq:manifold-ss-eq}  F(x) =0, \quad x\in \Mcal. \ee
To solve \eqref{eq:manifold-ss-eq}, we construct a projected SSN method as follows. Let $P_k \in \R^{n \times n}$ be the matrix representation of the orthogonal projection operator onto the tangent space $T_{x_k} \Mcal$ to $\Mcal$ at $x_k$. Starting from a point $x_k$ close enough to $x^*$, we consider the following projected SSN update
\be \label{eq:proj-ssn} (J_k P_k + \mu_k I ) d_k = - F(x_k) + r_k, \quad 
x_{k+1} = P_{\Mcal}(x_k + d_k), \ee
where $\mu_k = \|F(x_k)\|$ and the accuracy of \eqref{eq:proj-ssn} satisfies
    \be \label{eq:rk} \frac{\|r_k\|}{\mu_k} \leq L_3 \|F(x_k)\|^{q} \ee
    with $L_3 >0$ and $q \in (1,2]$.

\begin{remark}
    When $\Mcal = \R^n$, both $P_k$ and $P_\Mcal$ reduce to the identity mappings, and the SSN update \eqref{eq:proj-ssn} is exactly the same as \eqref{eq:ssn-it}. We note that the projected SSN update \eqref{eq:proj-ssn} is different from the Riemannian Newton method for solving $F(x) = 0, x\in \Mcal$ as the direction $d_k$ used in \eqref{eq:proj-ssn} may not lie in the tangent space $T_{x_k}\Mcal$, and $J_k$ may not be the Jacobian of the smooth restriction $F|_{\mathcal{M}}$.
\end{remark}

Before showing the convergence analysis, let us start with the following assumptions.
\begin{assum}
For a point $x^* \in X^*$, let $\Mcal$ be a smooth submanifold of $\R^n$ containing $x^*$.
\begin{itemize}
\item[{\rm (A1)}] There exists $b_1> 0$ such that $F$ is Lipschitz continuous on $\bar{B}(x^*,b_1)$ with modulus $L_2$. In addition, there exists a constant $L_1$ such that  
    \be \label{eq:quad-bound1} \| F(y) - F(x) - J(x)(y-x) \| \leq L_1 \|y-x\|^2, \;\; \forall x, y \in \Mcal \cap \bar{B}(x^*, b_1),\ee
    where $J(x)$ is any element of B-Jacobian of $F$ at $x$.
\item[{\rm (A2)}] There exists $b_2>0$ such that for all $x \in \Mcal \cap \bar{B}(x^*,b_2)$, it holds that 
\be \label{eq:eigen-J2} \|(J(x)P(x) + \mu(x)I)^{-1}\| \leq \mu(x)^{-1}, \ee
where $P(x)$ is the matrix representation of the orthogonal projection operator onto the tangent space $T_{x} \Mcal$.
\item[{\rm{(A3)}}] A local error bound condition holds for $F$ at $x^*$, i.e., there exist $b_3>0$ and $\gamma>0$ such that for all $x\in \Mcal \cap \bar{B}(x^*, b_3)$,
    \be \label{eq:errorbound} \|F(x)\| \geq \gamma \mathrm{dist}(x, X^* \cap \Mcal).\ee 
    \end{itemize}
For the ease of the subsequent analysis, we set
\begin{equation} \label{eq:def-b}
    b := \min\left\{b_2, b_3, \gamma b_1/(2L_1 + 2L_2 + 2\gamma + 2L_3L_2^q \gamma), 1, \gamma/(L_1 + L_2 + \gamma + L_3L_2^q \gamma)
 \right\}.
\end{equation}
% $b := \min\{b_1, b_2, b_3\}$. 
\end{assum}

\begin{remark} 
The assumption (A1) corresponds to the aforementioned smoothness of $F$ around $x^*$ on $\Mcal$. As shown in Theorems \ref{coro-pgm}, \ref{coro-drs}, and \ref{coro-alm}, the semismooth mappings $F_{\rm PGM} $, $F_{\rm DRS}$, and $F_{\rm ALM}$ satisfy (A1) with $\Mcal = \R^n$ under the corresponding smoothness conditions, e.g., partial smoothness.  More generally, for any locally Lipschitz semialgebraic mapping $F$, applying \cite[Corollary 9]{bolte07clarke} to each component of $F$, we obtain a finite partition of $\mathbb{R}^n$ into smooth submanifolds $\mathcal{M}_1, \dots, \mathcal{M}_k$ such that (A1) holds for every pair $(x^*,\, \mathcal{M}_i)$ such that $x^* \in \mathcal{M}_i$.
\end{remark}

\begin{remark} 
When $\Mcal =\R^n$, (A2) reduces to Assumption \ref{assum:invert-jacobian}, which holds if the conditions of Theorem \ref{thm:eigen} are satisfied. In addition, if $J(x)$ is positive semidefinite, then it is easy to verify that \eqref{eq:eigen-J} holds, and so does \eqref{eq:eigen-J2}.
Compared with \cite{fan2004inexact}, our local error bound assumption, (A2), is required to be satisfied on the submanifold $\Mcal$ instead of $\R^n$. For the case of $\Mcal = \R^n$, (A2) has served as an alternative condition to the nonsingularity of $F$ for the superlinear convergence in the smooth setting \cite{fan2004inexact} and has been shown to hold in many scenarios \cite{pang1987posteriori,luo1992linear,tseng2010approximation,zhou2017unified}. It follows from the Lipschitz continuity of $F$ that $\gamma \leq L_2$. 
\end{remark}

Note that the superlinear convergence results of the LM method for smooth $F$ and an inexact Newton method for set-valued $F$ are presented in \cite{fan2004inexact,fischer2002local}. However, these analyses either rely on the smoothness of $F$ or solving inner nonsmooth equations efficiently. Instead, we only require that $F$ is smooth on a smooth submanifold  $\Mcal$ around the solution set, which is weaker than the smoothness in \cite{fan2004inexact}.

\subsection{Local superlinear convergence}
We are going to show the superlinear convergence of the projected SSN method \eqref{eq:proj-ssn}. 
Note that the manifold $\Mcal$ is generally nonconvex. This requires us to be more careful on the analysis on the projection operator. As the smooth submanifold is generally proximally smooth \cite{clarke1995proximal,davis2020stochastic}, the projection $P_\Mcal$ in \eqref{eq:proj-ssn} is single-valued if $\|d_k\|$ is small enough. In addition, we need the following lemma on the Lipschitz-type and contraction-like properties of $P_\Mcal$ and a different variant of the smoothness condition \eqref{eq:quad-bound1}.

\begin{lemma} \label{lem:prox-smooth}
Suppose that $F$ satisfies (A1). Let $x$ be any element of $\Mcal \cap \bar{B}(x^*,b_1)$. 
\begin{itemize}
    \item[\rm (a)] There exists a $\beta > 0$ such that for any $d \in \R^{n}$,
\be \label{eq:proj-lip} \| P_\Mcal(x + d) - x - P(x)d \| \leq \beta \|d\|^2. \ee
    \item[\rm (b)] For any $y \in \R^n$, it holds that
    \be \label{eq:proj-nonexp} \| P_\Mcal(y) - x \| \leq \| P_\Mcal(y) - y + y - x \| \leq 2\|x - y\|.  \ee
    \item[\rm (c)] For any $y \in \Mcal \cap \bar{B}(x^*,b_1)$, it holds that
    \be \label{eq:quad-bound0}\begin{aligned}
& \| F(y) - F(x) - J(x)P(x)(y-x) \| \\
\leq &  \| F(y) - F(x) - J(x)(y-x)  \| + \|J(x)(I-P(x))(y-x) \| \\
\leq & L_1\|y-x\|^2 + \frac{n L_2}{2R} \|y-x\|^2,
\end{aligned}
\ee
where $R > 0$ is some constant only depending on $x^*, b_1$ and $\mathcal{M}$.
\end{itemize}
\end{lemma}
\begin{proof}
    It follows from \cite[Lemma 3.1]{absil2012projection} that the derivate $DP_{\Mcal}(x) = P(x)$. Then, (a) holds 
    by the boundedness of $\Mcal \cap \bar{B}(x^*,b_1)$. For (b), we have
    \[ \| P_\Mcal(y) - x \| \leq \| P_\Mcal(y) - y + y - x \| \leq 2\|x - y\|.  \]
    Denote by $N_x \Mcal$ the normal space to $\Mcal$ at $x$. By \cite{clarke1995proximal} and \cite[Lemma 3.1]{davis2020stochastic}, there exists a constant $R > 0$ such that for any $x,y \in \Mcal \cap \bar{B}(x^*,b_1)$ and $v \in N_{x}\Mcal$ with $N_x \Mcal$ being the normal space to $\Mcal$ at $x$, 
    $\iprod{v}{y-x} \leq \frac{\|v\|}{2R} \|y-x\|^2$.
   Consequently, (c) holds for any $x, y \in \Mcal \cap \bar{B}(x^*,b_1)$,
    \[ \begin{aligned}
    & \| F(y) - F(x) - J(x)P(x)(y-x) \| \\
    \leq &  \| F(y) - F(x) - J(x)(y-x)  \| + \|J(x)(I-P(x))(y-x) \| \\
    \leq & L_1\|y-x\|^2 + \frac{n L_2}{2R} \|y-x\|^2,
    \end{aligned}
    \]
    where the second inequality is from (A1). 
\end{proof}

With a slight abuse of the notation, we will continue using $L_1$ to denote $L_1 + nL_2/(2R)$. Throughout the following analysis, let $\tilde{X}^*:= X^* \cap \Mcal$ and define $\Pi_{\tilde{X}^*}(x) = \argmin_{y \in \tilde{X}^*} \;\; \|y -x\|$.
Then, we have the following relationship between the Newton direction and the distance of the iterates to the optimal set.

\begin{lemma} \label{lemma:dk}
 Under {\rm (A1), (A2)}, and {\rm (A3)}, there exists a constant $c_1>0$ such that if some $x_{k} \in \Mcal \cap \bar{B}\left(x^{*}, b / 2\right)$, then
\be
\left\|d_{k}\right\| \leq c_1 \operatorname{dist}\left(x_{k}, \tilde{X}^{*}\right)+\mu_{k}^{-1}\left\|r_{k}\right\|.
\ee
\end{lemma}
\begin{proof}
Pick any $\bar{x}_k\in\Pi_{\tilde{X}^*}(x^k)$. Since $x_{k} \in \Mcal \cap \bar{B}\left(x^{*}, b / 2\right)$, we have 
 $\bar{x}_k\in \bar{B}\left(x^{*}, b\right)$ by noting that 
$\left\|\bar{x}_k-x^{*}\right\| \leq\left\|\bar{x}_k-x_{k}\right\|+\left\|x_{k}-x^{*}\right\| \leq 2\left\|x_{k}-x^{*}\right\|\leq b.$
From (A1) and (A3), it follows that
\be \label{eq:est-muk}
\gamma \left\|\bar{x}_{k}-x_{k}\right\| \leq \mu_{k}=\left\|F_{k}\right\| \leq L_{2}\left\|\bar{x}_{k}-x_{k}\right\| .
\ee
Write $ v_{k}:=-F_{k}-\left(J_{k}P_k+\mu_{k} I\right)\left(\bar{x}_{k}-x_{k}\right). $
By invoking \eqref{eq:quad-bound1}, it follows that
\be \label{eq:vk}
\begin{aligned}
\left\|v_{k}\right\| & \leq \left\|F_{k}+J_{k}P_k\left(\bar{x}_{k}-x_{k}\right)\right\|+\mu_{k}\left\|\bar{x}_{k}-x_{k}\right\| \\
& \leq L_1\|\bar{x}_k-x_k\|^2+L_2\|\bar{x}_k-x_k\|^2
 \leq  (L_1+L_2) \left\|\bar{x}_{k}-x_{k}\right\|^{2}. 
\end{aligned}
\ee
Let $w_{k}:=d_{k}-\left(\bar{x}_{k}-x_{k}\right)$. Then, 
$\left(J_{k}P_k+\mu_{k} I\right) w_{k}=v_{k}+r_{k}$. Putting \eqref{eq:est-muk}, \eqref{eq:vk}, and \eqref{eq:eigen-J2} together gives
$$
\left\|w_{k}\right\| \leq \frac{\left\|v_{k}\right\|+\left\|r_{k}\right\|}{\mu_{k}} \leq \frac{L_{1}+ L_2}{\gamma}\left\|\bar{x}_{k}-x_{k}\right\|+\frac{\left\|r_{k}\right\|}{\mu_{k}}.
$$
 Note that $d_k=w_k+(\bar{x}_{k}-x_k)$. From the last inequality, we show that the desired inequality holds with 
 $c_{1}=\left(L_{1}+L_2\right) / \gamma +1$.
\end{proof}

The following lemma establishes the superlinear convergence of the distance from the iterates to the solution set $\tilde{X}^*$. 
%---------------------------------------------------------------------
\begin{lemma} \label{lemma:dist-suplinear}
 Under {\rm (A1), (A2)}, and {\rm (A3)}, there exists a constant $c_{2}>0$ such that if $x_{k}, x_{k+1} \in \Mcal \cap \bar{B}\left(x^{*}, b/2\right)$, then
\be \label{eq:dist-suplinear}
\operatorname{dist}\left(x_{k+1}, \tilde{X}^*\right) \leq c_{2} \operatorname{dist}\left(x_{k}, \tilde{X}^*\right)^q.
\ee
\end{lemma}
\begin{proof}
 Let $\bar{d}_k = - (J_kP_k + \mu_k I)^{-1} F_k$ be the exact  (regularized) SSN step. Then, 
\[
  \| F_k + J_kP_k \bar{d}_k \| = \mu_k\|\bar{d}_k\| \leq L_2 c_1 [{\rm dist}(x_k,\tilde{X}^*)]^2,
 \]
 where the inequality is from \eqref{eq:est-muk} and Lemma \ref{lemma:dk} with $r_k = 0$.
Note that
$ d_{k}=\bar{d}_k+\left(J_{k}P_k+\mu_{k} I\right)^{-1} r_{k}. $
Then, 
\be \label{eq:flinear}
\begin{aligned}
\left\|F_{k}+J_{k}P_k d_{k}\right\| &=\left\|F_{k}+J_{k}P_k \bar{d}_k+J_{k}P_k\left(J_{k}P_k +\mu_{k} I\right)^{-1} r_{k}\right\| \\
& \leq\left\|F_{k}+J_{k}P_k \bar{d}_k\right\|+\frac{L_{2}\left\|r_{k}\right\|}{\mu_{k}} \\
& \leq L_2 c_1 [{\rm dist}(x_k,\tilde{X}^*)]^2+\frac{L_{2}\left\|r_{k}\right\|}{\mu_{k}},
\end{aligned}
\ee
where the first inequality is from (A2).
From {\rm (A1)} and $b \le 1$, we have $\frac{\left\|r_{k}\right\|}{\mu_{k}}
\le L_3\|F(x_k)\|^q \le L_3L_2^q[{\rm dist}(x_k,\tilde{X}^*)]^q
\le L_3L_2^q{\rm dist}(x_k,\tilde{X}^*)$. Together with Lemma \ref{lemma:dk}, it follows that
\be \label{eq:bound-dk}
\left\|d_{k}\right\| \leq\left(c_{1}+L_3L_2^q \right){\rm dist}(x_k,\tilde{X}^*).
\ee
By \eqref{eq:proj-nonexp}, it holds that $\| P_{\Mcal}(x_k + d_k) - x^* \| \leq 2\|x_k - x^*\| + 2\|d_k\| \leq b_1$ due to \eqref{eq:def-b}. Combining inequalities \eqref{eq:quad-bound1} and \eqref{eq:flinear}-\eqref{eq:bound-dk} yields that
$$
\begin{aligned}
& \left\|F\left(P_\Mcal(x_{k}+d_{k})\right)\right\| \leq \left\|F_{k}+J_{k} (P_\Mcal(x_k + d_{k}) - x_k)\right\|+L_{1}\left\| P_\Mcal(x_k + d_{k}) - x_k \right\|^{2} \\ 
\leq &  \| F_k + J_k P_k d_k \| + L_2 \beta \|d_k\|^2 + L_1 (2\|d_k\|^2 + 2\beta\|d_k\|^4) \\
\leq & L_2 c_1 [{\rm dist}(x_k,\tilde{X}^*)]^2+\frac{L_{2}\left\|r_{k}\right\|}{\mu_{k}}+(2L_{1} + L_2 \beta + 2L_1\beta) \left(c_{1}+L_3L_2^q \right)^2[{\rm dist}(x_k,\tilde{X}^*)]^2 \\
\leq & \big(L_2 c_1+L_{2}^{q+1}L_3+(2L_{1} + L_2 \beta + 2L_1\beta)\left(c_{1}+L_3L_2^q\right)^{2}\big)[{\rm dist}(x_k,\tilde{X}^*)]^q.
\end{aligned}
$$
Thus, we have 
$$
\operatorname{dist}\left(P_{\Mcal}(x_{k}+d_{k}), \tilde{X}^*\right) \leq \gamma^{-1}\left\|F\left(P_{\Mcal}(x_{k}+d_{k})\right)\right\| \leq c_{2} \operatorname{dist}\left(x_{k}, \tilde{X}^*\right)^q
$$
with $c_{2}=\gamma^{-1}\big(L_2 c_1+L_{2}^{q+1} L_3+(2L_{1} + L_2 \beta + 2L_1\beta)\left(c_1+L_3L_2^q\right)^{2}\big)$. 
This completes the proof.
\end{proof}

When $x_0$ is close enough to $\tilde{X}^*$, all the iterates will stay in a small neighborhood of some point $\hat{x} \in \tilde{X}^*$. In addition, $x_k$ converges to $\hat{x}$ superlinearly. 
\begin{theorem} \label{thm:q-convergence}
 If $x_{0}$ is chosen sufficiently close to $\tilde{X}^*$ and the Assumptions {\rm (A1), (A2)}, and {\rm (A3)} hold at $x^* = \Pi_{\tilde{X}^*}(x_0)$,  then $x_k$ converges to some solution $\hat{x}$ of \eqref{prob:nonsmootheq} superlinearly.
\end{theorem}
\begin{proof} Let $\bar{x}_k = \Pi_{\tilde{X}^*}(x_k)$ for all $k$ and
$$
r=\min \left\{\frac{1}{2c_{2}^{1/(q-1)}}, \frac{b}{4\left[1 + c_1 2^{q-1}/(2^{q-1} - 1) + 2L_3L_2^q2^{q-1}/(2^{q-1} - 1)\right]}\right\},
$$
where the constants $c_{1}, c_{2}, L_3$ are defined previously. Firstly, we show by induction that if $x_{0} \in \bar{B}\left(\bar{x}_0, r\right)$, then $x_{k} \in \bar{B}\left(\bar{x}_0, b / 2\right)$ for all $k \ge 1$.
It follows from Lemma \ref{lemma:dk} that
$$
\begin{aligned}
\left\|x_{1}-\bar{x}_0\right\| &=\left\|P_{\Mcal}(x_{0}+d_{0})-\bar{x}_0\right\|
\leq 2\left\|x_{0}-\bar{x}_0\right\|+2 \| d_{0}\| \\
& \leq 2 \left\|x_{0}-\bar{x}_0\right\|+ 2 \left(c_{1}+L_3L_2^q\right)\left\|x_{0}-\bar{x}_{0}\right\| \\
& \leq 2\left(1+c_{1}+L_3L_2^q\right) r \leq b / 2.
\end{aligned}
$$
This gives $x_{1} \in \bar{B}\left(\bar{x}_0, b / 2\right)$. Suppose that $x_{i} \in \bar{B}\left(\bar{x}_0, b / 2\right)$ for $i=2, \cdots, k$. By Lemma \ref{lemma:dist-suplinear}  and $c_2 \geq 1$, we have
$$
\left\|x_{i}-\bar{x}_{i}\right\| \leq c_{2}\left\|x_{i-1}-\bar{x}_{i-1}\right\|^{q} \leq \cdots \leq c_{2}^{\frac{q^i-1}{q-1}}\left\|x_{0}-\bar{x}_0\right\|^{q^{i}} \leq r 2^{-q^{i}}.
$$
It then follows from the definition of $r$ that 
$$
\begin{aligned}
\left\|x_{k+1}-\bar{x}_0\right\| & \leq\left\|x_{1}-\bar{x}_0\right\|+ 2 \sum_{i=1}^{k}\left\|d_{k}\right\| \\
& \leq 2\left(1+c_{1}+L_3L_2^q\right) r+ 2 \left(c_{1}+L_3L_2^q\right) \sum_{i=1}^{k}\left\|x_{i}-\bar{x}_{i}\right\| \\
% & \leq\left(1+c_{1}+L_3L_2^q\right) r+ r\left(c_{1}+L_3L_2^q\right) \sum_{i=1}^{k} 2^{-q^i}\\
% & \leq\left(1+c_{1}+L_3L_2^q\right) r+ r\left(c_{1}+L_3L_2^q\right) \sum_{i=1}^{k} 2^{-(q-1)i - 1}\\
 & \leq 2 \left(1+c_{1}+L_3L_2^q\right) r+ 2 r\left(c_{1}+L_3L_2^q\right) \sum_{i=1}^{\infty} \left(2^{-(q-1)}\right)^i \\
& \leq 2 r\left[ 1 + c_1 2^{q-1}/(2^{q-1} - 1) + 2L_3L_2^q2^{q-1}/(2^{q-1} - 1) \right] \\
& \leq b / 2.
\end{aligned}
$$
This gives $x_{k+1} \in \bar{B}\left(\bar{x}_0, b / 2\right)$. Thus, if $x_{0}$ is chosen sufficiently close to $\tilde{X}^*$, then $x_{k}$ lie in $\bar{B}\left(\bar{x}_0, b / 2\right)$ for all $k\ge 1$. It follows from Lemma \ref{lemma:dist-suplinear} that
$
\sum_{k=0}^{\infty} \operatorname{dist}\left(x_{k}, \tilde{X}^*\right)<+\infty.
$
% This, implies that
% $
% \sum_{k=0}^{\infty}\left\|d_{k}\right\|<+\infty.
% $
Thus, for any $\epsilon > 0$, there exists a $K> 0$ such that for any $K \leq k_1 \leq k_2$,
\be \label{eq:Cauchy} \|x_{k_2} - x_{k_1} \| \leq \sum_{i = k_1}^{k_2 - 1} \|x_{i+1} - x_{i}\| \leq 2 \sum_{i = k_1}^{k_2 - 1} \|d_i\| \leq 2(c_1 + L_3L_2^q) \sum_{i = k_1}^{k_2 - 1} {\rm dist}(x_{i}, \tilde{X}^*) \leq 
\epsilon,  \ee
where the third inequality is from \eqref{eq:bound-dk}.
This means that $\{x_k\}_{k\geq 0}$ is a Cauchy sequence and hence converges to some point $\hat{x} \in \tilde{X}^*$. Taking $k_1 = k$ such that $({\rm dist}(x_{k}, \tilde{X}^*))^{q-1} \leq \rho$ with $\rho < 1$ and $k_2$ goes to infinity in \eqref{eq:Cauchy} lead to
\[ \begin{aligned}
\|x_{k} - \hat{x} \| & \leq 2(c_1 + L_3L_2^q) \sum_{i = k}^{\infty} {\rm dist}(x_i, \tilde{X}^*) \\
& \leq 2(c_1 + L_3L_2^q)  \sum_{i = 0}^{\infty} \rho^i {\rm dist}(x_{k}, \tilde{X}^*) \leq \frac{2(c_1 + L_3L_2^q)}{1-\rho}  {\rm dist}(x_{k}, \tilde{X}^*).
\end{aligned}
 \]
This, together with Lemma \ref{lemma:dist-suplinear}, implies that
\[ \|x_{k+1} - \hat{x}\| \leq \frac{2(c_1 + L_3L_2^q)}{1-\rho}  {\rm dist}(x_{k+1}, \tilde{X}^*) \leq \mathcal{O}({\rm dist}(x_{k}, \tilde{X}^*)^q) \leq \mathcal{O}(\| x_k - \hat{x} \|^q), \]
which means that $x_{k}$ converges to the solution $\hat{x}$ superlinearly. We complete the proof.
\end{proof}

\section{Numerical verification} \label{sec:num}
In this section, we conduct numerical experiments on the Lasso problem and basis pursuit to validate our theory of locally superlinear convergence rate. Note that the two sufficient conditions (A1) and (A2) are numerically verified to be satisfied in these two examples.

\subsection{The Lasso problem}
We first  empirically check BD-regularity and SC on the Lasso problem: $ \min_{x \in \R^n}  \frac{1}{2}\|Ax - b\|_2^2 + \lambda \|x\|_1$, 
where $A \in \R^{m\times n}$, $b \in \R^m$ and $\lambda >0$ is the regularization parameter. Let $x$ be a solution and $T_1$ be its support, i.e., $T_1 := \{ i: x_i \ne 0 \}$. Let $T_2:= \{ i: x_i = 0, \; |(A^\top (Ax - b))_i| = \lambda \}$ and $S = T_1 \cup T_2$.
 Define $A_S = A(:, S)$ as the submatrix corresponding to $S$. As shown in Examples \ref{eg} and \ref{eg:lasso-drs}, the BD-regularity conditions of the natural residual \eqref{eq:pgm} and the DRS residual \eqref{eq:drs} hold if and only if $A_S^\top A_S$ is positive definite. From \cite{stella2017forward}, as long as $T_2 = \emptyset$, SC holds. It is worth mentioning that SC is empirically observed to be satisfied for the Lasso problem for randomly generated $A$ \cite{hale2008fixed}.

Let us show a numerical setting to verify the superlinear convergence under SC.  Consider a matrix $A \in \R^{m\times n}$ where  two columns, whose indices are denoted by  $i_1 =$ ind(1) and $i_2 =$ ind(2),  are the same (one can also generalize to multiple columns with linear dependence). It is easy to check that if  $x^*$ is a solution, any $v$ from the following set is also a solution:
$$\{v \in \R^n \,:\, v_{i_1} + v_{i_2} = x^*_{i_1} + x^*_{i_2}, \, |v_{i_1}| + |v_{i_2}| = |x^*_{i_1}| + |x^*_{i_2}|,\, v_i = x^*_i \mathrm{~if~} i\ne i_1\mathrm{~or~} i_2\}. $$ Hence, whenever $x_{i_1}^*$ or $x_{i_2}^*$ is nonzero, $x^*$ is not an isolated solution.  Specifically, we  construct a Gaussian random $A \in \R^{m\times n}$ and a vector $b$ with $m=64$ and $n=128$ using the following MATLAB commands:
\begin{lstlisting}[frame=single]
 A = randn(m,n); u = sprandn(n,1,0.1); 
ind = find(u>1e-7); A(:,ind(1)) = A(:,ind(2)); b = A*u;
\end{lstlisting}
where ``randn'', ``sprandn'', and ``find'' are built-in functions in MATLAB. 
The parameter $\lambda$ is set to $10^{-3}$. 

The natural residual based SSN method is utilized to solve the corresponding problem. Let $x^*$ be the obtained solution, $T_2$ and $S$ be defined as above with $x_i^*$ considered as $0$ when $|x_i^*| < 10^{-7}$. Since the minimum eigenvalue  $\lambda_{\min}(A_S^\top A_S)$ is $7.9\times 10^{-16}$,  BD regularity does not hold from Example \ref{eg}.  
 Although $F(x)$ becomes linear in a small neighborhood of $x^*$ according to the theory  in Section \ref{sec:convergence}, the Jacobian is singular and the standard Newton method without regularization may have the numerical issues in solving a singular linear equation. However, the SSN method \eqref{eq:ssn} reduces to the Newton method with regularization. Because the minimal value of $\{ ||(A^\top (Ax^*-b))_i| - \lambda |: i \in T_2 \}$ is $1.1\times 10^{-4}$, SC holds.  Our theory also shows that  the iterates will converge superlinearly. Figure \ref{fig:ssn} shows the iteration history, which matches our theoretical results. 

\begin{figure}[h]
     \centering
    \subfloat[$\|F(x^k)\|_2$]{
    \includegraphics[width = 0.33\textwidth,height=0.25\textwidth]{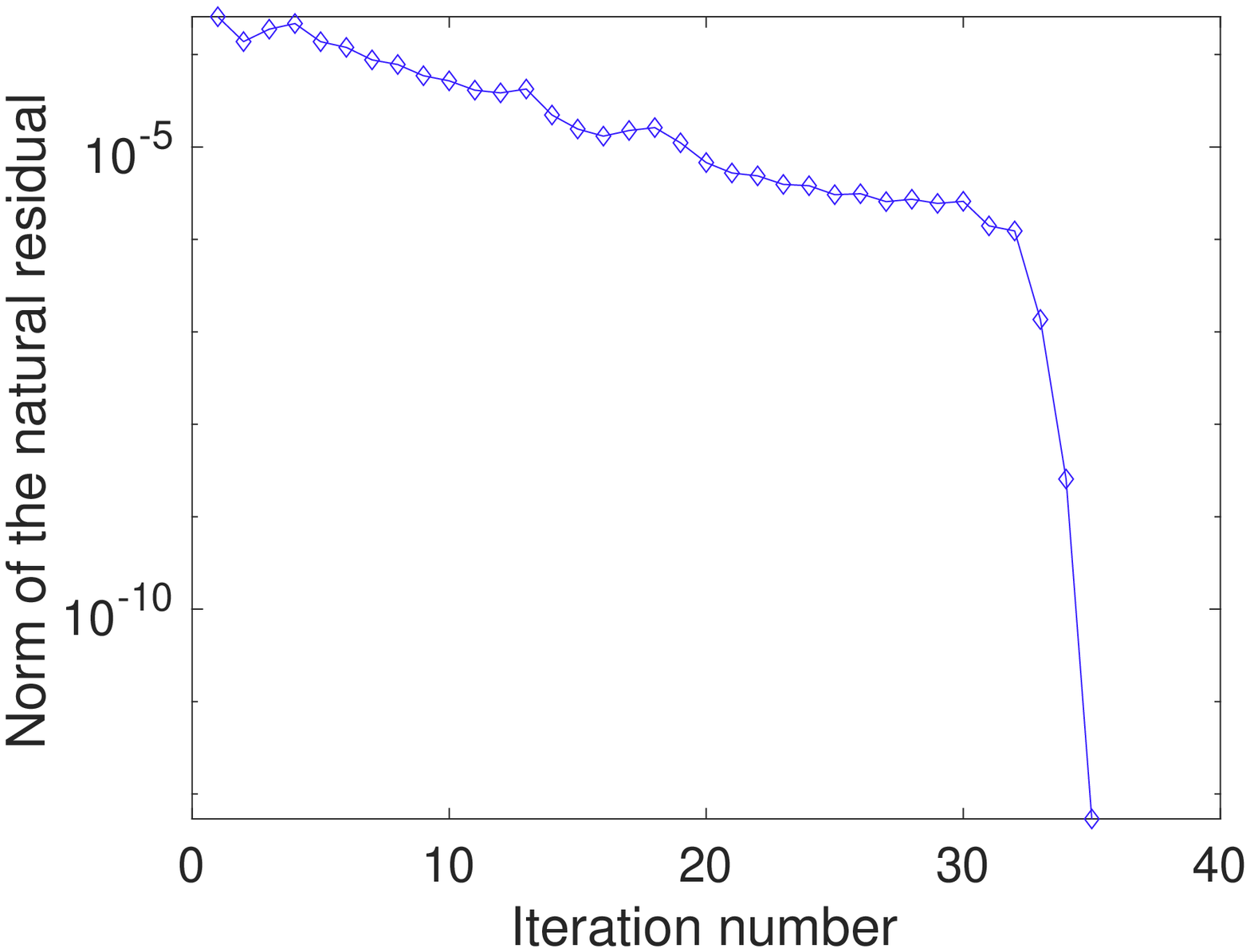}}
    \subfloat[The  solution]{
    \includegraphics[width = 0.33\textwidth,height=0.25\textwidth]{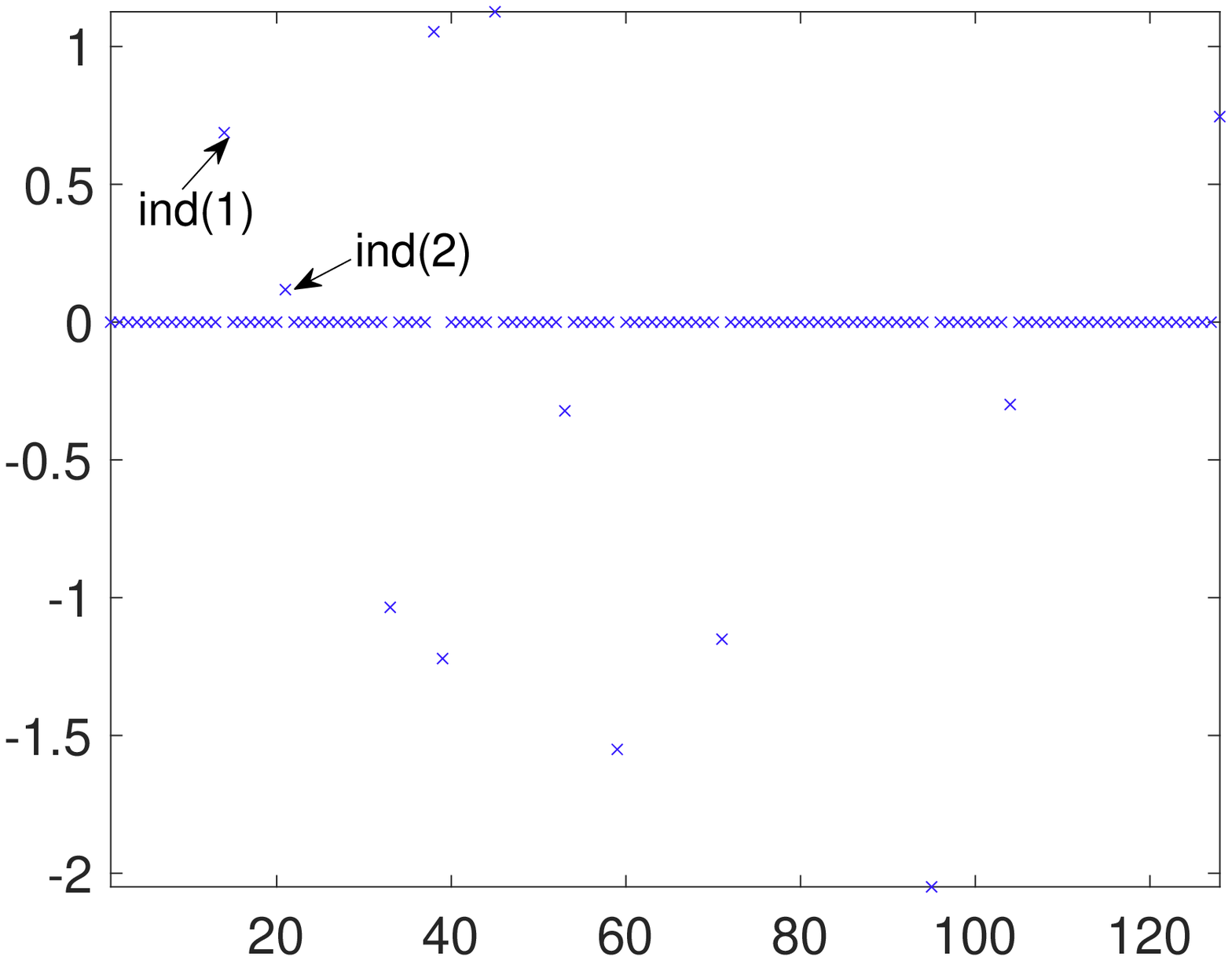}}
    \caption{Natural residual based SSN method for solving Lasso problem. For the last iteration point $x^*$, SC holds, while BD-regularity does not.  (a). Locally superlinear convergence rate is observed. (b). Both $x_{i_1}^*$ and $x_{i_2}^*$  in $x^*$ are non-zero. It implies that $x^*$ is not an isolated solution.}
    \label{fig:ssn}
\end{figure}

\subsection{Basis pursuit}

Consider the basis pursuit problem \eqref{bp-primal1} and its dual problem \eqref{bp-dual1}. Let $f(x) = \delta_{\{x:Ax=b\}}(x)$ and $h(x) = \|x\|_1$.
By the equivalence between the alternating direction method of multipliers and the DRS method \cite{yan2016self,li2018semismooth}, the dual solution is $y^* = A(\mathrm{prox}_{th}(z^*) - z^*)/t$ with $z^*$ being the root of $F$. 

Consider a Gaussian random $A \in \R^{m\times n}$ with $m=64$ and $n=128$. Set $\lambda = 10^{-3}$ and use the following commands in Matlab to generate $A$ and $b$:
\begin{lstlisting}[frame=single]
 A = randn(m,n); u = sprandn(n,1,0.1); ind = find(u>1e-7);
 A(:,ind(1))= A(:,ind(2)); [Q,~] = qr(A',0); A = Q'; b = A*u;
\end{lstlisting}
where ``randn'', ``sprandn'',  ``find'' and ``qr'' are built-in functions in MATLAB. Here, we orthogonalize $A$ for the ease of implementations of $\prox_{tf}$. 
It holds that $\mathrm{prox}_{tf}(x) = x - A^\top(Ax - b)$. If $x^*$ is a solution, then any $v$ from the set,
$$\{v \in \R^n \,:\, v_{i_1} + v_{i_2} = x^*_{i_1} + x^*_{i_2}, \, |v_{i_1}| + |v_{i_2}| = |x^*_{i_1}| + |x^*_{i_2}|,\, v_i = x^*_i \mathrm{~if~} i\ne i_1\mathrm{~or~} i_2\} $$ with $i_1 =$ ind(1) and $i_2 =$ ind(2), is also a solution. Hence, whenever $x_{i_1}^*$ or $x_{i_2}^*$ is not zero, $x^*$ is not an isolated solution.

Figure \ref{fig:ssn-drs-BP} shows the results of using the DRS residual based SSN method.
In this example, 
 the SC holds at $(x^*,y^*)$ by following Example \ref{eq:bp-sc}. To be specific, let $\Omega(x^*)=\{i: x^* = 0\}$ denote the index set where the entries in $x^*$ are zero. All elements in $\{1 - |A^\top y^*|_i : i\in \Omega(x^*)\}$ are exactly zero. For all $i$ in the index set $\Omega^c(x^*)$, the minimal value of $|1 - |A^\top y^*|_i|$ is 0.2. This means that either $1-|A^\top y^*|_i$ or $x_i^*$ is zero but not both. We plot the entries of $x^*$ and $1-|A^\top y^*|$ in (b) and (c), respectively. The BD-regularity condition is not satisfied due to the nonisolatedness of the solutions. The superlinear convergence rate is observed, which validates our theoretical arguments.

\begin{figure}[h]
    \begin{center}
    \subfloat[$\|F(z^k)\|_2$]{
    \includegraphics[width = 0.32\textwidth,height=0.25\textwidth]{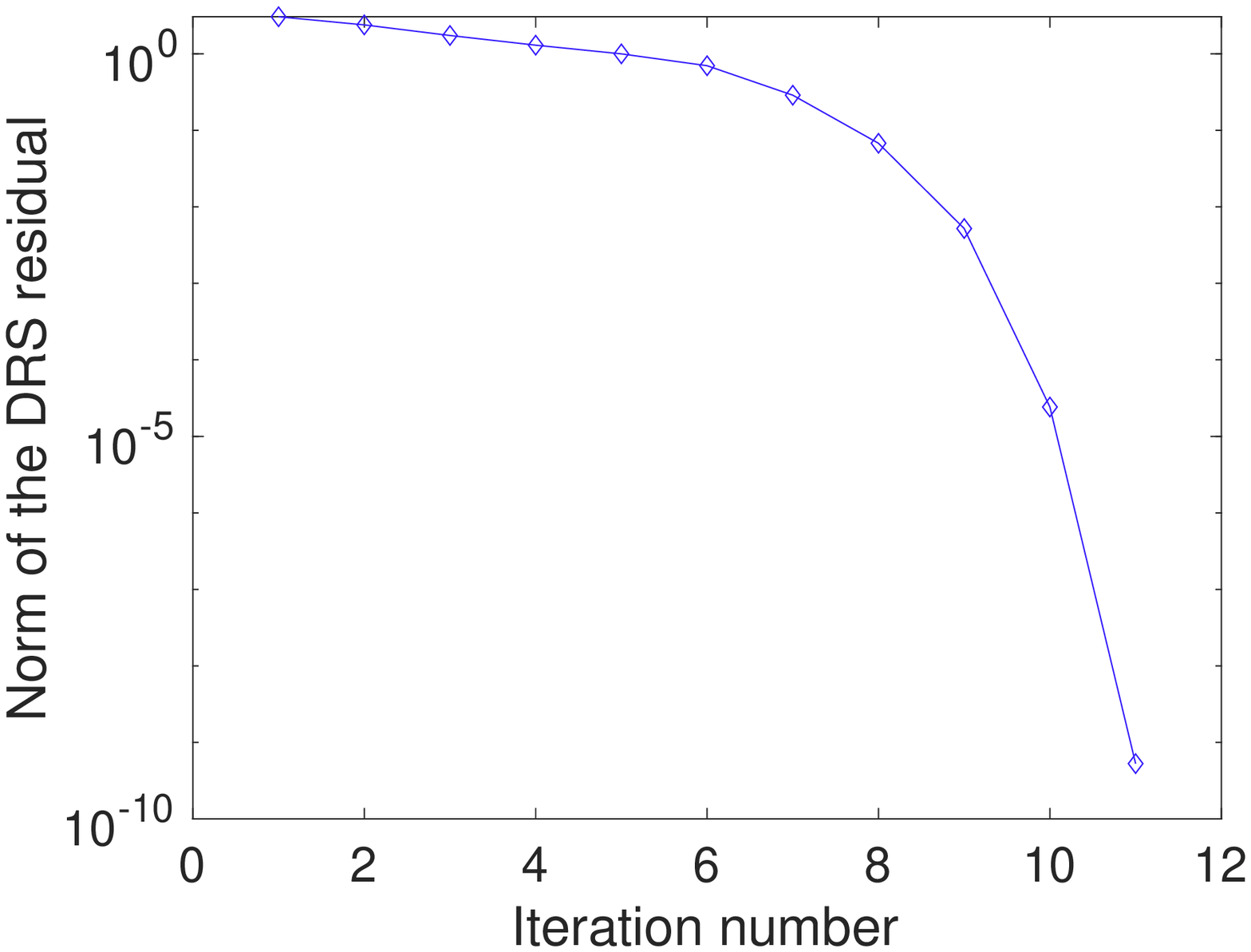}}
    \subfloat[The solution]{
    \includegraphics[width = 0.32\textwidth,height=0.25\textwidth]{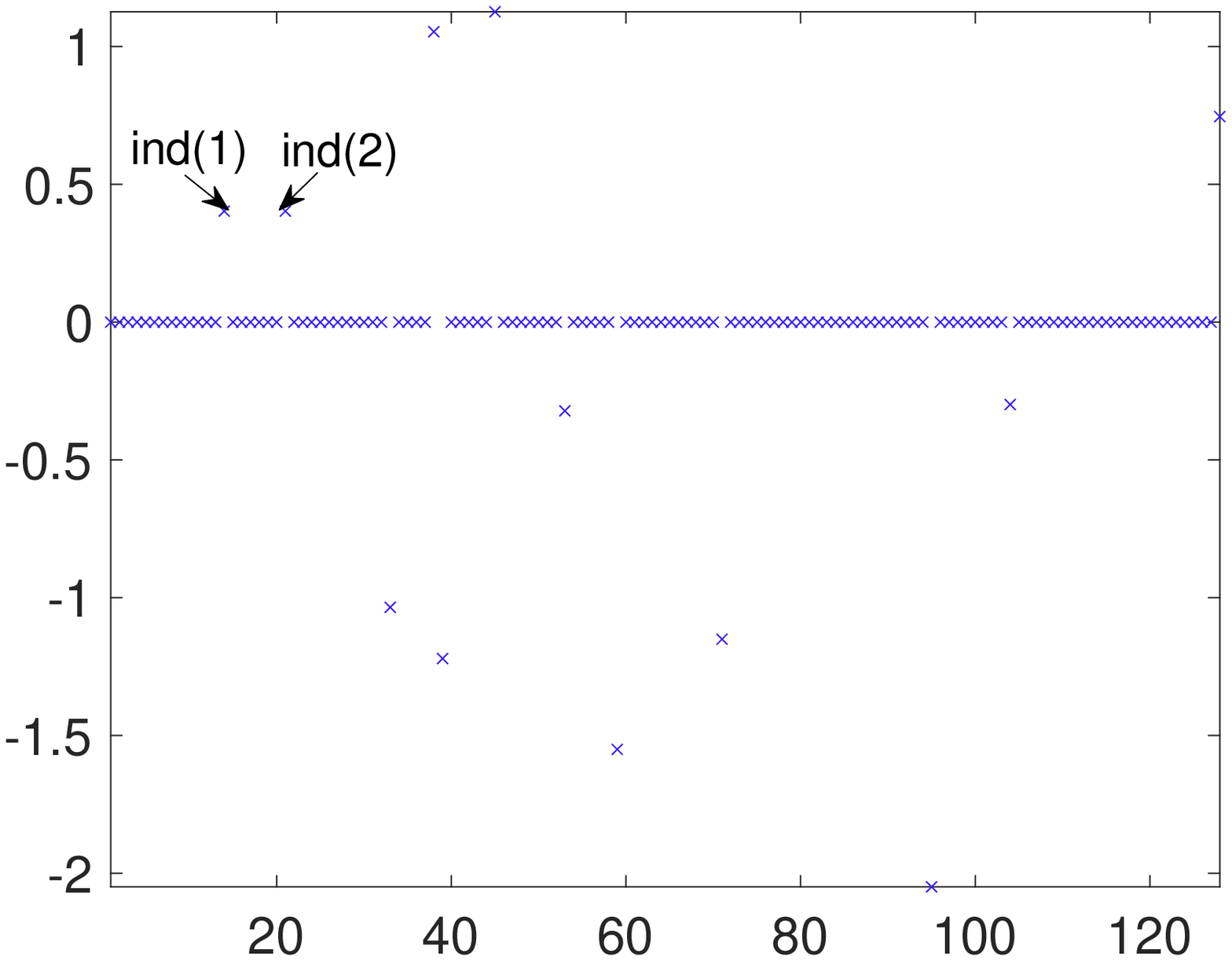}}
    \subfloat[$1-|A^\top y^*|$ ]{
    \includegraphics[width = 0.32\textwidth,height=0.25\textwidth]{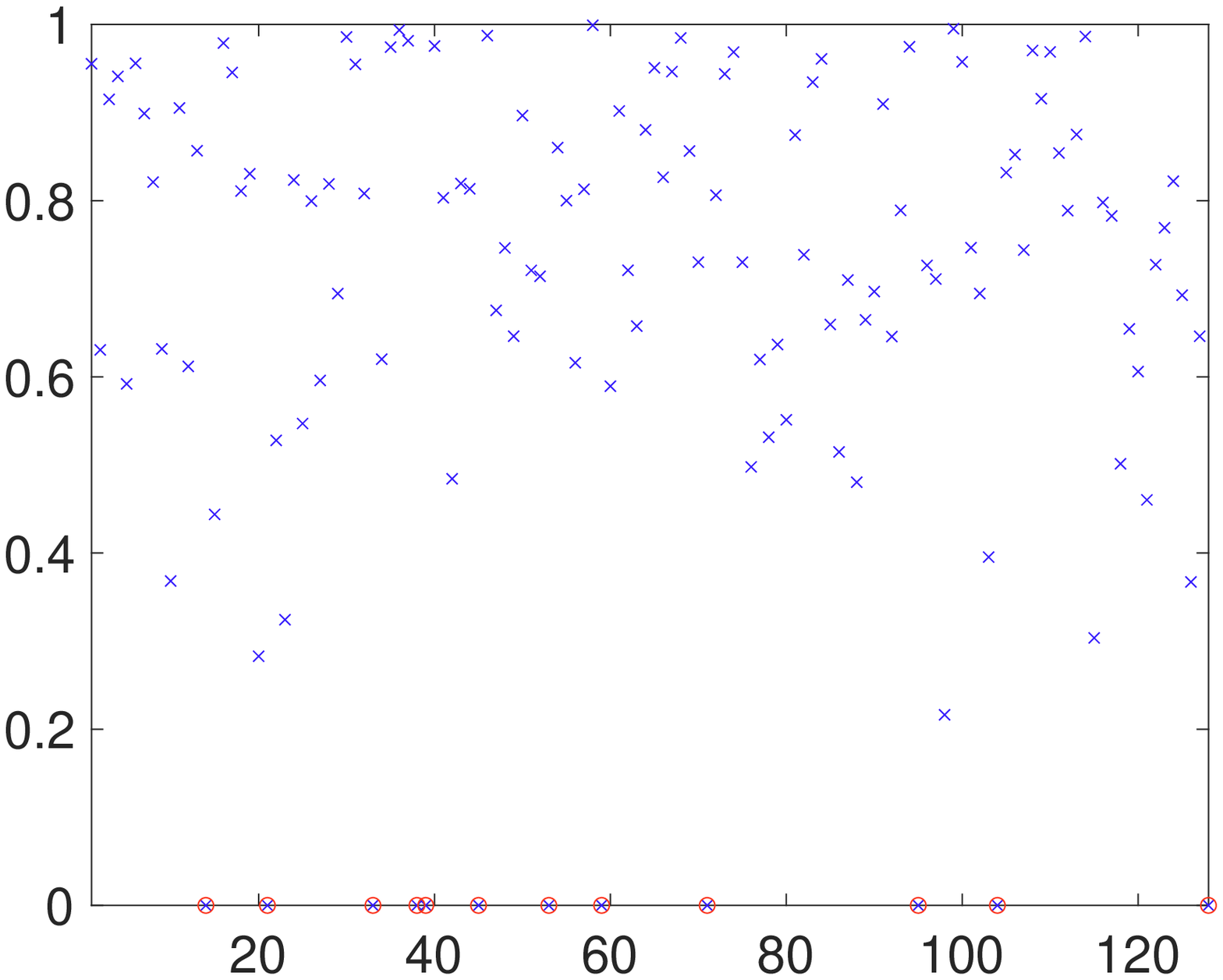}}
    \end{center}
    \caption{DRS residual based SSN method for solving basis pursuit problem. For $x^* = \prox_{th}(z^*)$ with $z^*$ being the last iteration point, the SC holds, while the BD-regularity condition does not.  (a). Locally superlinear convergence rate is observed. (b). Both $x_{i_1}^*$ and $x_{i_2}^*$ in $x^*$ are non-zero. It implies that $x^*$ is not an isolated solution and the BD-regularity condition fails. (c) The red ``o'' indicates the indices of the non-zero elements in $x^*$ and the blue ``x'' represents the values of $1-|A^\top y^*|$, where  $y^* = A(x^* -z^*)/t$ is the solution of the dual problem. We see that the SC holds at the solution pair $(x^*, y^*)$ according to Example \ref{eq:bp-sc}.}
    \label{fig:ssn-drs-BP}
\end{figure}

\section{Conclusion} 

We study the convergence of the SSN methods for a large class of nonlinear equations derived from first-order type methods for solving
composite optimization problems. Existing literature on the superlinear convergence analysis of SSN methods predominantly hinges on either BD-regularity or local smoothness.
For BD-regularity, we introduce equivalent characterizations to facilitate easier verification. In the case of local smoothness, we demonstrate that it can be inferred from two types of conditions. The first condition combines partial smoothness with  SC, while the second condition requires the second subderivative to be generalized quadratic. We further illustrate that numerous nonconvex and nonsmooth functions, commonly encountered in practical applications, meet these stipulated conditions. Moreover, we establish the superlinear convergence of a projected SSN method, predicated on local smoothness on a submanifold and the local error bound condition. Figure \ref{fig:compare-BD} shows the relationships among BD-regularity, the EB condition, the isolatedness of the stationary point, and the second-order sufficient condition.

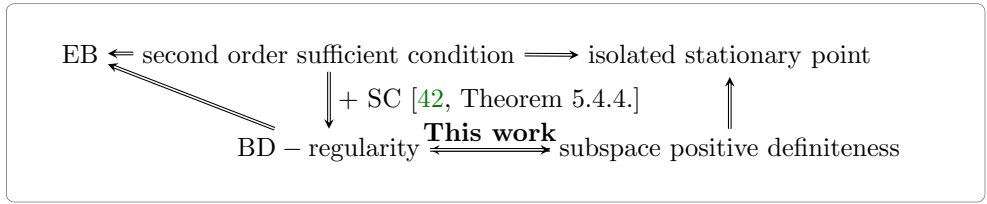
\begin{figure} 
\begin{center}
	\begin{tcolorbox}[colback = white, boxrule = {0.2pt}]
		\begin{tikzpicture}
			\matrix (m) [matrix of math nodes,row sep=2em,column sep=1em,minimum width=1em]
			{
			{\rm EB} & {\rm second~order~sufficient~condition} & {\rm isolated~stationary~point} \\
			\;	& {\rm BD}-{\rm regularity} & {\rm subspace~positive~definiteness}\\};
			\path[-stealth]
			(m-1-2) edge [double] node [right] {} (m-1-1)
			(m-2-2) edge [double] node [right] {} (m-1-1)
			(m-1-2)	edge [double] node [right] {+ SC \cite[Theorem 5.4.4.]{milzarek2016numerical}} (m-2-2)
			(m-1-2) edge [double] node [left] {} (m-1-3)
			(m-2-3) edge [double] node [left] {} (m-2-2) 
			(m-2-2) edge [double] node [above] {{\bf This work}} (m-2-3) 
			(m-2-3) edge [double] node [left] {} (m-1-3);
		\end{tikzpicture}
	\end{tcolorbox}
\end{center} \caption{Implications of different concepts of optimality conditions.}
\label{fig:compare-BD}
\end{figure}

\typeout{}
\bibliographystyle{siamplain}
\bibliography{references}

\end{document}